\documentclass[10pt,a4paper,english,english]{amsart}
\usepackage{mathpazo}
\usepackage[T1]{fontenc}
\usepackage[latin9]{inputenc}
\usepackage{color}
\usepackage{amsthm}
\usepackage{graphicx}
\usepackage{amssymb}

\numberwithin{equation}{section} 
\numberwithin{figure}{section} 
\theoremstyle{plain}
 \theoremstyle{definition}
 \newtheorem*{defn*}{Definition}
\newtheorem{thm}{Theorem}[section]
  \theoremstyle{definition}
  \newtheorem{defn}[thm]{Definition}
  \theoremstyle{remark}
  \newtheorem{rem}[thm]{Remark}
  \theoremstyle{plain}
  \newtheorem{lem}[thm]{Lemma}
  \theoremstyle{plain}
  \newtheorem{prop}[thm]{Proposition}
  \theoremstyle{plain}
  \newtheorem{assumption}[thm]{Assumption}
  \theoremstyle{remark}
  \newtheorem{notation}[thm]{Notation}
  \theoremstyle{plain}
  \newtheorem{cor}[thm]{Corollary}
 \theoremstyle{definition}
  \newtheorem{example}[thm]{Example}

\usepackage{babel}
\addto\extrasenglish{\providecommand{\fg}{\ifdim\lastskip>\z@\unskip\fi~\frqq}}

\begin{document}

\title{Quantum revivals in two degrees of freedom integrable systems : the
torus case}

\author{Olivier Lablée}

\date{1 September 2010}
\begin{abstract}
The paper deals with the semi-classical behaviour of quantum dynamics
for a semi-classical completely integrable system with two degrees
of freedom near Liouville regular torus. The phenomomenon of wave
packet revivals is demonstrated in this article. The framework of
this paper is semi-classical analysis (limit : $h\rightarrow0$).
For the proofs we use standard tools of real analysis, Fourier analysis
and basic analytic number theory. 
\end{abstract}
\maketitle

\section{Introduction}

\subsection{Motivation}

In quantum physics, on a Riemannian manifold $(M,g)$ the evolution
of an initial state $\psi_{0}\in L^{2}(M)$ is given by the famous
Schrödinger equation :

\begin{eqnarray*}
ih\frac{\partial\psi(t)}{\partial t}=P_{h}\psi(t);\;\,\psi(0)=\psi_{0}.\end{eqnarray*}
Here $h>0$ is the semi-classical parameter and the operator $P_{h}\,:\, D\left(P_{h}\right)\subset L^{2}\left(M\right)\rightarrow L^{2}\left(M\right)$
is $h$-pseudo-differential operator (for example $P_{h}=-\frac{h^{2}}{2}\Delta_{g}+V$).
In the case of dimension 1 or for completely integrable systems, we
can describe the semi-classical eigenvalues of the Hamiltonian $P_{h}$
and by linearity we can write the solutions of the Schrödinger equation.
Nevertheless, the behaviour of the solutions when the times $t$ evolves
in larges times scales remains quite mysterious. 

In dimension 1, the dynamics in the regular case and for elliptic
non-degenerate singularity have been the subject of many research
in physics \textbf{{[}Av-Pe{]}, {[}LAS{]}, {[}Robi1{]}}, \textbf{{[}Robi2{]},
{[}BKP{]}}, \textbf{{[}Bl-Ko{]} }and, more recently in mathematics
\textbf{{[}Co-Ro{]}}, \textbf{{[}Rob{]}, {[}Pau1{]}}, \textbf{{[}Pau2{]},
{[}Lab2{]}}. The strategy to understand the long times behaviour of
dynamics is to use the spectrum of the operator $P_{h}$. In the regular
case, the spectrum of $P_{h}$ is given by the famous Bohr-Sommerfeld
rules (see for example \textbf{{[}He-Ro{]},} \textbf{{[}Ch-VuN{]}},
\textbf{{[}Col{]}}) : in first approximation, the spectrum of $P_{h}$
in a compact set is a sequence of real numbers with a gap of size
$h.$ The classical trajectories are periodic and supported on elliptic
curves. Always in dimension 1, in the case of hyperbolic singularity
we have a non-periodic trajectory. The spectrum near this singularity
is more complicated than in the regular case. In \textbf{{[}Lab3{]}
}we have an explicit description of the spectrum for an one-dimesional
pesudo-differential operator near a hyperbolic non-degenerate singularity.
The article \textbf{{[}Lab4{]}} deals with the quantum dynamics for
the hyperbolic case. So, in dimension 1, we get the full and fractionnals
revivals phenomenon (see \textbf{{[}Av-Pe{]}, {[}LAS{]}, {[}Robi1{]}},
\textbf{{[}Robi2{]}, {[}BKP{]}}, \textbf{{[}Bl-Ko{]}, {[}Co-Ro{]}},
\textbf{{[}Rob{]}, {[}Pau1{]} }for the elliptic case and see \textbf{{[}Lab4{]}
}or\textbf{ {[}Pau2{]}} for the the hyperbolic case). For an initial
wave packets localized in energy, the dynamics follows the classical
motion during short time, and, for large time, a new period $T_{rev}$
for the quantum dynamics appears : the initial wave packets form again
at $t=T_{rev}$. 

Physicists R. Blhum, A. Kostelecky and B. Tudose are interested in
the case of the dimension 2 (see \textbf{{[}BKT{]}}). Our paper presents
some accurate results on the time evolution for a generical semi-classical
completely integrable system of dimension 2 with mathematical proofs.

\subsection{Results and paper organization}

Here the quantum Hamiltonian is of the type $P_{h}=F\left(P_{1},P_{2}\right)$
where $F$ is a real polynomial of two variables and $P_{1},P_{2}$
are semi-classical one dimensional harmonic oscillators (see section
2 for details). By a diffeomorphism this Hamiltonian is less particular
than it seems to be, since it gives the spectrum of any completely
integrable system with two degrees of freedom near regular torus or
around elliptic singularity \textbf{{[}VuN{]}}. Therefore, the Hamiltonian
study leads to a study more or less general but which is not obvious
in dimension 2. In this paper, we consider an initial state $\psi_{0}$
localized near some regular Liouville torus of energies $\left(E_{1},E_{2}\right)$
and we study the associated quantum dynamics. To understand the behaviour
of dynamics, we interested in the evolution of the autocorrelation
: \[
\mathbf{a}(t)=\left|\left\langle \psi(t),\psi_{0}\right\rangle _{L^{2}(\mathbb{R}^{2})}\right|.\]
Due to the simple nature of the Hamiltonian operator the autocorrelation
function can be write as a serie :\[
\mathbf{a}(t)=\left|\sum_{n=0}^{+\infty}\sum_{m=0}^{+\infty}\left|a_{n,m}\right|^{2}e^{-i\frac{t}{h}F\left(\tau_{n},\mu_{m}\right)}\right|,\]
where $\tau_{n}=\omega_{1}h\left(n+\frac{1}{2}\right),\;\mu_{m}=\omega_{2}h\left(m+\frac{1}{2}\right)$
are eigenvalues of the one-dimensionnal harmonic oscillators $P_{1},P_{2}$.
The sequence $\left(a_{n,m}\right)_{n,m}$ is just the decomposition
of the initial vector $\psi_{0}$ on the Hermitte's eigenbasis of
$L^{2}(\mathbb{R}^{2})$. 

Most of the paper (section 3 and 4) consists in estimating and analyzing
the function $\mathbf{a}(t)$ for large times scales ($t\leq1/h^{s}$
with various $s>0$). We use Taylor's formula to expand the phase
term $tF\left(\tau_{n},\mu_{m}\right)/h$ in the variables $(n,m);$
first in linear order (section 3), then to quadratic order (section
4).

In the section 3, we study the linear approximation $\mathbf{a_{1}}(t)$
(see definition 3.6) of the autocorrelation function, valid up on
a time scale $\left[0,1/h^{\alpha}\right]$ where $0<\alpha<1$. The
dynamics depends strongly on the diophantin properties of the classical
periods $T_{cl_{1}},T_{cl_{2}}$. If the fraction $T_{cl_{1}}/T_{cl_{2}}$
is commensurate (in this case the classical Hamiltonian flow is \foreignlanguage{english}{$T_{cl}$}-periodic)
we can describe accurately the behaviour of the dynamics on a classical
period $\left[0,T_{cl}\right]$ (see theorem 3.12). In opposite, if
the fraction $T_{cl_{1}}/T_{cl_{2}}$ is a bad approximation by rationals
(we suppose $T_{cl_{1}}/T_{cl_{2}}$ is Roth number) the autocorrelation
function collapse in the set $\left]0,T_{s}\right]$ where $T_{s}$
is order of $1/h^{s}$ (see theorem 3.24). For large time we use the
continuous fraction expansion of $T_{cl_{1}}/T_{cl_{2}}$ to analyze
some possible periods for linear approximation $\mathbf{a_{1}}(t)$
(see theorem 3.35).

In the last section, we use the quadradic approximation $\mathbf{a_{2}}(t)$
(see definition 4.6) of the autocorrelation function, valid up on
a time scale $\left[0,1/h^{\beta}\right]$ where $\beta>1$. In this
quadradic approximation appear three revivals periods $T_{rev_{1}},T_{rev_{2}}$
and $T_{rev_{12}}$ of order $1/h$ depending on the Hessian matrix
of the function $F$ at the point $\left(E_{1},E_{2}\right)$. If
we suppose $T_{rev_{1}},T_{rev_{2}}$ and $T_{rev_{12}}$ are commensurate,
we can proove and analyze the revivals phenomenon (see theorem 4.16
and corollary 4.17). In the last subsection we compute the modulus
of the revival coefficients (see theorem 4.19).

\section{General points}

\subsection{Some basic facts on semi-classical analysis }

To explain quickly the philosophy of semi-classical analysis, starts
by an example : for a real number $E>0$; the equation\[
-\frac{h^{2}}{2}\Delta_{g}\varphi=E\varphi\]
(where $\Delta_{g}$ denotes the Laplace-Beltrami operator on a Riemaniann
manifold $(M,g)$) admits the eigenvectors $\varphi_{k}$ as solution
if\[
-\frac{h^{2}}{2}\lambda_{k}=E.\]
Hence if $h\rightarrow0^{+}$ then $\lambda_{k}\rightarrow+\infty$.
So there exists a correspondence between the semi-classical limit
($h\rightarrow0^{+}$) and large eigenvalues. 

The asympotic of large eigenvalues for the Laplace-Beltrami operator
$\Delta_{g}$ on a Riemaniann manifold $(M,g)$, or more generally
for a pseudo-differential operator $P_{h}$, is linked to a symplectic
geometry : the phase space geometry. This is the same phenomenon between
quantum mechanics (spectrum, operators algebra) and classical mechanics
(length of periodic geodesics, symplectic geometry). For more details
see for example the survey\textbf{ {[}Lab1{]}}.

\subsection{Quantum dynamics and autocorrelation function}

For a quantum Hamiltonian $P_{h}\,:\, D\left(P_{h}\right)\subset\mathcal{H}\rightarrow\mathcal{H}$,
where $\mathcal{H}$ is a Hilbert space, the Schrödinger dynamics
is governed by the Schrödinger equation : \begin{eqnarray*}
ih\frac{\partial\psi(t)}{\partial t}=P_{h}\psi(t).\end{eqnarray*}
With the functional calculus, we can reformulate this equation with
the unitary group $U(t)=\left\{ e^{-i\frac{t}{h}P_{h}}\right\} _{t\in\mathbb{R}}.$
Indeed, for a initial state $\psi_{0}\in\mathcal{H}$, the evolution
is given by : \begin{eqnarray*}
\psi(t)=U(t)\psi_{0}\in\mathcal{H}.\end{eqnarray*}
We now introduce a simple tool to understand the behaviour of the
vector \foreignlanguage{english}{$\psi(t)$} : this tool is a quantum
analog of the Poincaré return function :
\begin{defn*}
The quantum return functions of the operator $P_{h}$ and for an initial
state $\psi_{0}$ is defined by : \[
\mathbf{r}(t):=\left\langle \psi(t),\psi_{0}\right\rangle _{\mathcal{H}};\]
and the autocorrelation function is defined by :\[
\mathbf{a}(t):=\left|\mathbf{r}(t)\right|=\left|\left\langle \psi(t),\psi_{0}\right\rangle _{\mathcal{H}}\right|.\]

\end{defn*}
The previous function measures the return on the initial state $\psi_{0}$.
This function is the overlap of the time dependent quantum state $\psi(t)$
with the initial state $\psi_{0}.$ Since the initial state $\psi_{0}$
is normalized, the autocorrelation function takes values in the compact
set $[0,1].$

\subsection{The Hamiltonian of our model}

For our study, the quantum Hamiltonian is the operator :\[
P_{h}:=F\left(P_{1},P_{2}\right)\]
where $F$ is a polynomial of $\mathbb{R}\left[X,Y\right]$ which
does not depend on the paramater $h$; $P_{1}$ and $P_{2}$ are the
Weyl-quantization of the classical one dimensional harmonic oscillator
: \[
p_{j}\left(x_{1},\xi_{1},x_{2},\xi_{2}\right)=\omega_{j}\left(x_{j}^{2}+\xi_{j}^{2}\right)/2\]
 with $\omega_{1},\omega_{2}>0$. It is well know that the Hermitte
functions $\left(e_{n,m}\right)_{n,m}:=\left(e_{n}\otimes e_{m}\right)_{n,m\in\mathbb{N}^{2}}$
is a Hilbert basis of the space $L^{2}(\mathbb{R}^{2})$. Let us consider
for all integers ($n,m)$ the eigenvalues of $P_{1}$ and $P_{2}$
: \[
\tau_{n}:=\omega_{1}h\left(n+\frac{1}{2}\right),\;\mu_{m}:=\omega_{2}h\left(m+\frac{1}{2}\right);\]
so, we get immediatly that for all integers ($n,m)$ \foreignlanguage{english}{\[
F\left(P_{1},P_{2}\right)\left(e_{n}\otimes e_{m}\right)=F\left(\tau_{n},\mu_{m}\right)\left(e_{n}\otimes e_{m}\right).\]
}

\subsection{The autocorrelation function rewritten in a eigenbasis}

Now, for a initial vector $\psi_{0}={\displaystyle \sum_{n,m\in\mathbb{N}^{2}}a_{n,m}e_{n,m}}$
we have for all $t\geq0$

\begin{eqnarray*}
\psi(t)=\left(e^{-i\frac{t}{h}F\left(P_{1},P_{2}\right)}\right)\left({\displaystyle \sum_{n,m\in\mathbb{N}^{2}}a_{n,m}e_{n,m}}\right)={\displaystyle \sum_{n,m\in\mathbb{N}^{2}}a_{n,m}e^{-i\frac{t}{h}F\left(\tau_{n},\mu_{m}\right)}e_{n,m}}\end{eqnarray*}
so, for all $t\geq0$ we obtain

\[
\mathbf{r}(t)=\sum_{n=0}^{+\infty}\sum_{m=0}^{+\infty}\left|a_{n,m}\right|^{2}e^{-i\frac{t}{h}F\left(\tau_{n},\mu_{m}\right)}\]
and\[
\mathbf{a}(t)=\left|\sum_{n=0}^{+\infty}\sum_{m=0}^{+\infty}\left|a_{n,m}\right|^{2}e^{-i\frac{t}{h}F\left(\tau_{n},\mu_{m}\right)}\right|.\]
The aim of this paper is to study this sum, but unfortunately this
function is too difficult to be understood immediatly.

\subsection{Strategy to study the autocorrelation function }

The strategy for simplify the sum function $t\mapsto\mathbf{a}(t)$,
performed by the physicists (\textbf{{[}Av-Pe{]}, {[}LAS{]}, {[}Robi1{]}},
\textbf{{[}Robi2{]}, {[}BKP{]}}, \textbf{{[}Bl-Ko{]}}) is the following
:
\begin{enumerate}
\item we define a initial vector $\psi_{0}={\displaystyle {\displaystyle \sum_{n,m\in\mathbb{N}^{2}}a_{n,m}e_{n,m}}}$
localized near some regular Liouville torus of energies $\left(E_{1},E_{2}\right)$
: consequently the sequence \foreignlanguage{english}{$(a_{n,m})_{n,m\in\mathbb{N}^{2}}$}
is localized close to a pair of quantum numbers $n_{0},m_{0}$ (depends
on $h$ and on the Liouville torus $(E_{1},E_{2})$.
\item Next, the idea is to expand by a Taylor formula's the eigenvalues
\foreignlanguage{english}{$F\left(\tau_{n},\mu_{m}\right)$} around
the Liouville torus $(E_{1},E_{2})$ :\[
F\left(\tau_{n},\mu_{m}\right)=\]
\[
F\left(\tau_{n_{0}},\mu_{m_{0}}\right)+h\omega_{1}\left(n-n_{0}\right)\frac{\partial F}{\partial X}\left(\tau_{n_{0}},\mu_{m_{0}}\right)+h\omega_{2}\left(m-m_{0}\right)\frac{\partial F}{\partial Y}\left(\tau_{n_{0}},\mu_{m_{0}}\right)\]
\[
+\frac{1}{2}\omega_{1}^{2}h^{2}\left(n-n_{0}\right)^{2}\frac{\partial^{2}F}{\partial X^{2}}\left(\tau_{n_{0}},\mu_{m_{0}}\right)+\frac{1}{2}\omega_{2}^{2}h^{2}\left(m-m_{0}\right)^{2}\frac{\partial^{2}F}{\partial Y^{2}}\left(\tau_{n_{0}},\mu_{m_{0}}\right)\]
\[
+\omega_{1}\omega_{2}h^{2}\left(n-n_{0}\right)\left(m-m_{0}\right)\frac{\partial^{2}F}{\partial X\partial Y}\left(\tau_{n_{0}},\mu_{m_{0}}\right)+\cdots\]
(here $\tau_{n_{0}},\mu_{m_{0}}$ is the closest pair of eigenvalue
to the pair $E_{1},E_{2}$). As a consequence we get for all $t\geq0$\[
\mathbf{a}(t)=\left|\sum_{n,m\in\mathbb{N}^{2}}\left|a_{n,m}\right|^{2}e^{-it\left[\omega_{1}\left(n-n_{0}\right)\frac{\partial F}{\partial X}\left(\tau_{n_{0}},\mu_{m_{0}}\right)+\cdots+\omega_{1}\omega_{2}h\left(n-n_{0}\right)\left(m-m_{0}\right)\frac{\partial^{2}F}{\partial X\partial Y}\left(\tau_{n_{0}},\mu_{m_{0}}\right)+\cdots\right]}\right|.\]

\item And, for small values of $t$, the first approximation of the autocorrelation
function $\mathbf{a}(t)$ is the function \[
\mathbf{a}_{1}(t):=\left|\sum_{n,m\in\mathbb{N}^{2}}\left|a_{n,m}\right|^{2}e^{-it\left[\omega_{1}\left(n-n_{0}\right)\frac{\partial F}{\partial X}\left(\tau_{n_{0}},\mu_{m_{0}}\right)+\omega_{2}\left(m-m_{0}\right)\frac{\partial F}{\partial Y}\left(\tau_{n_{0}},\mu_{m_{0}}\right)\right]}\right|;\]
and for larger values of $t$, the order 2-approximation is given
by\[
\mathbf{a}_{2}(t):=\left|\sum_{n,m\in\mathbb{N}^{2}}\left|a_{n,m}\right|^{2}e^{-it\left[\omega_{1}\left(n-n_{0}\right)\frac{\partial F}{\partial X}\left(\tau_{n_{0}},\mu_{m_{0}}\right)+\cdots+\omega_{1}\omega_{2}h\left(n-n_{0}\right)\left(m-m_{0}\right)\frac{\partial^{2}F}{\partial X\partial Y}\left(\tau_{n_{0}},\mu_{m_{0}}\right)\right]}\right|.\]

\end{enumerate}
In section 3, we study in details the function $t\mapsto\mathbf{a}_{1}(t)$
and $t\mapsto\mathbf{a}_{2}(t)$ in section 4.

\subsection{Choice of an initial state }

Let us define an initial vector $\psi_{0}={\displaystyle \sum_{n,m\in\mathbb{N}^{2}}a_{n,m}e_{n,m}}$
localized near a regular Liouville torus of energies $E:=\left(E_{1},E_{2}\right)$
where $E_{1}\in\left[0,1\right]$ and $E_{2}\in\left[0,1\right]$.
\begin{defn}
Let us consider the quantum integers $n_{0}=n_{0}(h,E_{1})$ and $m_{0}=m_{0}(h,E_{2})$
defined by\[
n_{0}:=\textrm{arg}\min_{n}\left|\tau_{n}-E_{1}\right|;\; m_{0}:=\textrm{arg}\min_{m}\left|\mu_{m}-E_{2}\right|.\]
\end{defn}
\begin{rem}
Without loss of generality, we may suppose that the integers $n_{0}$
and $m_{0}$ are unique.
\end{rem}
The integer $n_{0}$ (resp. $m_{0}$) is the eigenvalues index of
the operator from the family $P_{1}$ (resp. $P_{2}$) the closest
to the real number $E_{1}$ (resp. $E_{2}$). Since the spectral gap
of $P_{1}$ (resp. $P_{2}$) is equal to $\omega_{1}h$ (resp. $\omega_{2}h$
) we have, for $h\rightarrow0$ : $n_{0}\sim\frac{E_{1}}{\omega_{1}h};\; m_{0}\sim\frac{E_{2}}{\omega_{2}h}.$ 

Now, we can give definition of our initial state :
\begin{defn}
Let us consider the sequence $\left(a_{n,m}\right)_{n,m\in\mathbb{Z}^{2}}=\left(a_{n,m}(h)\right)_{n,m\in\mathbb{Z}^{2}}$
defined by :\[
a_{n,m}:=K_{h}\chi\left(\frac{\tau_{n}-\tau_{n_{0}}}{h^{\delta_{1}^{\prime}}},\frac{\mu_{m}-\mu_{m_{0}}}{h^{\delta_{2}^{\prime}}}\right)=K_{h}\chi\left(\omega_{1}\frac{n-n_{0}}{h^{\delta_{1}^{\prime}-1}},\omega_{2}\frac{m-m_{0}}{h^{\delta_{2}^{\prime}-1}}\right);\]
where the function $\chi$ is non null, non-negative and belong ot
the space $\mathcal{S}(\mathbb{R}^{2})$. The parameters $\left(\delta_{1}^{\prime},\delta_{2}^{\prime}\right)\in$$]0,1[^{2}.$
We also denote\[
K_{h}:=\left\Vert \chi\left(\frac{\tau_{n}-\tau_{n_{0}}}{h^{\delta_{1}^{\prime}}},\frac{\mu_{m}-\mu_{m_{0}}}{h^{\delta_{2}^{\prime}}}\right)\right\Vert _{\ell^{2}(\mathbb{N}^{2})}.\]

\end{defn}
Let us detail this choice :
\begin{enumerate}
\item the term \foreignlanguage{english}{$\chi\left(\frac{\tau_{n}-\tau_{n_{0}}}{h^{\delta_{1}^{\prime}}},\frac{\mu_{m}-\mu_{m_{0}}}{h^{\delta_{2}^{\prime}}}\right)$
localize around the torus $\left(E_{1},E_{2}\right)$ (for technical
reason we localize around the }closest eigenvalues\foreignlanguage{english}{
to $\left(E_{1},E_{2}\right)$.}
\selectlanguage{english}%
\item Constants\foreignlanguage{english}{ $\delta_{1}^{\prime}$ and $\delta_{2}^{\prime}$
are coefficients for dilate the function $\chi$ (the reason to take
$0<\delta_{j}^{\prime}<1$ is the following : it is the unique way
to have a non-trivial localization (not tend to $\left\{ 0\right\} $)
and a localization larger the spectral $h^{\delta_{j}^{\prime}}\gg h$). }
\end{enumerate}
\selectlanguage{english}%
So, clearly the sequence $\left(a_{n,m}\right)_{n,m}\in\ell^{2}(\mathbb{Z}^{2})$.
Now, let us evaluate the constant of normalization $K_{h}$; start
by the :
\begin{lem}
For a function $\varphi\in\mathcal{S}(\mathbb{R}^{2})$ and $(\varepsilon_{1},\varepsilon_{2})\in\left]0,1\right]^{2}$
then we have uniformly for $(u_{1},u_{2})\in\mathbb{R}^{2}$ : \[
\sum_{\ell,s\in\mathbb{Z}^{2},\,\left|\ell+u_{1}\right|\geq\frac{1}{2},\,\left|s+u_{2}\right|\geq\frac{1}{2}}\left|\varphi\left(\frac{\ell+u_{1}}{\varepsilon_{1}},\frac{s+u_{2}}{\varepsilon_{2}}\right)\right|=O(\varepsilon_{1}^{\infty}+\varepsilon_{2}^{\infty}).\]
\end{lem}
\begin{proof}
We see easily that, uniformly for $(u_{1},u_{2})\in\mathbb{R}^{2}$
we have \[
\sum_{\ell,s\in\mathbb{Z}^{2},\,\left|\ell+u_{1}\right|\geq\frac{1}{2},\,\left|s+u_{2}\right|\geq\frac{1}{2}}\left|\varphi\left(\frac{\ell+u_{1}}{\varepsilon_{1}},\frac{s+u_{2}}{\varepsilon_{2}}\right)\right|=O(1).\]
Next\[
\sum_{\ell,s\in\mathbb{Z}^{2},\,\left|\ell+u_{1}\right|\geq\frac{1}{2},\,\left|s+u_{2}\right|\geq\frac{1}{2}}\left|\varphi\left(\frac{\ell+u_{1}}{\varepsilon_{1}},\frac{s+u_{2}}{\varepsilon_{2}}\right)\right|\]
\[
\leq\sum_{\ell,s\in\mathbb{Z}^{2},\,\left|\ell+u_{1}\right|\geq\frac{1}{2},\,\left|s+u_{2}\right|\geq\frac{1}{2}}\left(\frac{\ell+u_{1}}{\varepsilon_{1}}\right)^{2N}\left|\varphi\left(\frac{\ell+u_{1}}{\varepsilon_{1}},\frac{s+u_{2}}{\varepsilon_{2}}\right)\right|\frac{\varepsilon_{1}^{2N}}{(\ell+u_{1})^{2N}}\]
\[
\leq\varepsilon_{1}^{2N}4^{N}\sum_{\ell,s\in\mathbb{Z}^{2}}\left(\frac{\ell+u_{1}}{\varepsilon_{1}}\right)^{2N}\left|\varphi\left(\frac{\ell+u_{1}}{\varepsilon_{1}},\frac{s+u_{2}}{\varepsilon_{2}}\right)\right|.\]
And, similary we have\[
\sum_{\ell,s\in\mathbb{Z}^{2},\,\left|\ell+u_{1}\right|\geq\frac{1}{2},\,\left|s+u_{2}\right|\geq\frac{1}{2}}\left|\varphi\left(\frac{\ell+u_{1}}{\varepsilon_{1}},\frac{s+u_{2}}{\varepsilon_{2}}\right)\right|\]

\[
\leq\varepsilon_{2}^{2N}4^{N}\sum_{\ell,s\in\mathbb{Z}^{2}}\left(\frac{s+u_{2}}{\varepsilon_{2}}\right)^{2N}\left|\varphi\left(\frac{\ell+u_{1}}{\varepsilon_{1}},\frac{s+u_{2}}{\varepsilon_{2}}\right)\right|.\]
To conclude the proof, we apply that to the functions $\psi(x,y):=x^{2N}\varphi(x,y)$
and $\psi(x,y):=y^{2N}\varphi(x,y)$.
\end{proof}
An obvious consequence of this lemma is the following result :
\begin{prop}
We get \[
K_{h}=\frac{1}{\sqrt{\mathfrak{F}\left(\chi^{2}\right)(0,0)}h^{\frac{\delta_{1}^{\prime}+\delta_{2}^{\prime}-2}{2}}}+O\left(h^{\infty}\right);\]
hence $\left\Vert a_{n,m}\right\Vert _{\ell^{2}(\mathbb{N}^{2})}=1+O(h^{\infty}).$\end{prop}
\begin{proof}
By the Poisson formula and the lemma above we get the equality :\textit{\[
\sum_{n,m\in\mathbb{Z}^{2}}\chi^{2}\left(\omega_{1}\frac{n-n_{0}}{h^{\delta_{1}^{\prime}-1}},\omega_{2}\frac{m-m_{0}}{h^{\delta_{2}^{\prime}-1}}\right)=h^{\delta_{1}^{\prime}+\delta_{2}^{\prime}-2}\sum_{\ell,s\in\mathbb{Z}^{2}}\mathfrak{F}\left(\chi^{2}\right)\left(-\ell\frac{h^{\delta_{1}^{\prime}-1}}{\omega_{1}},-s\frac{h^{\delta_{2}^{\prime}-1}}{\omega_{2}}\right)\]
\[
=h^{\delta_{1}^{\prime}+\delta_{2}^{\prime}-2}\left[\mathfrak{F}\left(\chi^{2}\right)(0,0)+\sum_{\ell,s\in\mathbb{\mathbb{Z}}^{2},\,\left|\ell\right|+\left|s\right|\geq1}\mathfrak{F}\left(\chi^{2}\right)\left(-\ell\frac{h^{\delta_{1}^{\prime}-1}}{\omega_{1}},-s\frac{h^{\delta_{2}^{\prime}-1}}{\omega_{2}}\right)\right]\]
\[
=h^{\delta_{1}^{\prime}+\delta_{2}^{\prime}-2}\mathfrak{F}\left(\chi^{2}\right)(0,0)+O\left(h^{\infty}\right).\]
}Now, with the basic equality\textit{\[
\sum_{n,m\in\mathbb{N}^{2}}\chi^{2}\left(\frac{n-n_{0}}{h^{\delta_{1}^{\prime}-1}},\frac{m-m_{0}}{h^{\delta_{2}^{\prime}-1}}\right)=h^{\delta_{1}^{\prime}+\delta_{2}^{\prime}-2}\mathfrak{F}\left(\chi^{2}\right)(0,0)+O\left(h^{\infty}\right)\]
\[
-\sum_{n=-\infty}^{-1}\sum_{m=-\infty}^{+\infty}\chi^{2}\left(\frac{n-n_{0}}{h^{\delta_{1}^{\prime}-1}},\frac{m-m_{0}}{h^{\delta_{2}^{\prime}-1}}\right)-\sum_{n=0}^{+\infty}\sum_{m=-\infty}^{-1}\chi^{2}\left(\frac{n-n_{0}}{h^{\delta_{1}^{\prime}-1}},\frac{m-m_{0}}{h^{\delta_{2}^{\prime}-1}}\right).\]
}and with the lemma above we see easily that \[
\sum_{n=-\infty}^{-1}\sum_{m=-\infty}^{+\infty}\chi^{2}\left(\frac{n-n_{0}}{h^{\delta_{1}^{\prime}-1}},\frac{m-m_{0}}{h^{\delta_{2}^{\prime}-1}}\right)=O(h^{\infty}),\]
\foreignlanguage{english}{\[
\sum_{n=0}^{+\infty}\sum_{m=-\infty}^{-1}\chi^{2}\left(\frac{n-n_{0}}{h^{\delta_{1}^{\prime}-1}},\frac{m-m_{0}}{h^{\delta_{2}^{\prime}-1}}\right)=O(h^{\infty}).\]
Finally we get : }\textit{\[
\left\Vert \chi\left(\frac{n-n_{0}}{h^{\delta_{1}^{\prime}-1}},\frac{m-m_{0}}{h^{\delta_{2}^{\prime}-1}}\right)\right\Vert _{\ell^{2}(\mathbb{N}^{2})}^{2}=\mathfrak{F}\left(\chi^{2}\right)(0,0)h^{\delta_{1}^{\prime}+\delta_{2}^{\prime}-2}+O\left(h^{\infty}\right);\]
}hence \[
K_{h}=\frac{1}{\sqrt{\mathfrak{F}\left(\chi^{2}\right)(0,0)}h^{\frac{\delta_{1}^{\prime}+\delta_{2}^{\prime}-2}{2}}}+O\left(h^{\infty}\right).\]
For finish, we write\textit{\[
\left\Vert a_{n,m}\right\Vert _{\ell^{2}(\mathbb{N}^{2})}^{2}=K_{h}^{2}\sum_{n,m\in\mathbb{Z}^{2}}\left|\chi\left(\omega_{1}\frac{n-n_{0}}{h^{\delta_{1}^{\prime}-1}},\omega_{2}\frac{m-m_{0}}{h^{\delta_{2}^{\prime}-1}}\right)\right|^{2}\]
\[
=K_{h}^{2}h^{\delta_{1}^{\prime}+\delta_{2}^{\prime}-2}\left[\mathfrak{F}\left(\chi^{2}\right)(0,0)+O(h^{\infty})\right]=1+O(h^{\infty}).\]
}
\end{proof}

\subsection{Technical interlude : the set $\Delta$}

In this subsection, we introduce the set $\Delta\subset\mathbb{N}^{2}$,
this set is useful for making approximation for autocorrelation function.
Start by the definition : 
\begin{defn}
Let us define the set of integers $\Delta=\Delta(h,E_{1},E_{2})$
by :

\[
\Delta:=\left\{ \left(n,m\right)\in\mathbb{N}^{2};\,\left|\tau_{n}-\tau_{n_{0}}\right|\leq\omega_{1}h^{\delta_{1}}\textrm{ and }\left|\mu_{m}-\mu_{m_{0}}\right|\leq\omega_{2}h^{\delta_{2}}\right\} \]
\[
=\left\{ \left(n,m\right)\in\mathbb{N}^{2};\,\left|n-n_{0}\right|\leq h^{\delta_{1}-1}\textrm{ and }\left|m-m_{0}\right|\leq h^{\delta_{2}-1}\right\} \]
where $0<\delta_{i}<1$; and we define the set $\Gamma=\Gamma(h,E_{1},E_{2})$
by :\[
\Gamma:=\mathbb{N}^{2}-\Delta.\]

\end{defn}
We have the following usefull lemma :
\begin{lem}
If we suppose for all $i\in\{1,2\},$ $\delta_{i}^{\prime}>\delta_{i}$
then we have\[
{\displaystyle \sum_{n,m\in\Gamma}\left|a_{n,m}\right|^{2}}=O(h^{\infty}).\]
\end{lem}
\begin{proof}
The starting point is the following inequality :\[
{\displaystyle \sum_{n,m\in\Gamma}\left|a_{n,m}\right|^{2}}\leq{\displaystyle \sum_{n,m\in\mathbb{Z}^{2},\,\left|n-n_{0}\right|>h^{\delta_{1}-1}}\left|a_{n,m}\right|^{2}}+\sum_{n,m\in\mathbb{Z}^{2},\,\left|m-m_{0}\right|>h^{\delta_{2}-1}}\left|a_{n,m}\right|^{2}.\]
Since the function $\chi^{2}$ is in the space $\mathcal{S}(\mathbb{R}^{2}),$
for all integer $N\geq1$ we have\foreignlanguage{english}{\[
{\displaystyle \sum_{n,m\in\mathbb{Z}^{2}}\left(\frac{n-n_{0}}{h^{\delta_{1}^{\prime}-1}}\right)^{2N}}\left|a_{n,m}\right|^{2}+{\displaystyle \sum_{n,m\in\mathbb{Z}^{2}}\left(\frac{m-m_{0}}{h^{\delta_{2}^{\prime}-1}}\right)^{2N}}\left|a_{n,m}\right|^{2}=O(1).\]
}Without loss generality, we may suppose that $n_{0}=m_{0}=0$. Next
we write\foreignlanguage{english}{\[
{\displaystyle \sum_{n,m\in\mathbb{Z}^{2},\,\left|n\right|>h^{\delta_{1}-1}}\left|a_{n,m}\right|^{2}}=h^{2N(\delta_{1}^{\prime}-1)}{\displaystyle \sum_{n,m\in\mathbb{Z}^{2},\,\left|n\right|>h^{\delta_{1}-1}}\left|a_{n,m}\right|^{2}}\left(\frac{n}{h^{\delta_{1}^{\prime}-1}}\right)^{2N}\frac{1}{n^{2N}}\]
}

\selectlanguage{english}%
\[
=O\left(h^{2N(\delta_{1}^{\prime}-\delta_{1})}\right).\]
\foreignlanguage{english}{In a similar way, we get }\[
{\displaystyle \sum_{n,m\in\mathbb{Z}^{2},\,\left|m\right|>h^{\delta_{2}-1}}\left|a_{n,m}\right|^{2}}=O\left(h^{2N(\delta_{2}^{\prime}-\delta_{2})}\right);\]
\foreignlanguage{english}{because $\delta_{i}^{\prime}>\delta_{i}$,
this implies ${\displaystyle \sum_{n,m\in\Gamma}\left|a_{n,m}\right|^{2}}=O(h^{\infty}),$
so we prove the lemma.}
\end{proof}
\selectlanguage{english}%

\section{Order 1 approximation : classical periods}

\subsection{Introduction}

In this section, we use a Taylor's formula to expand the phase term
$tF\left(\tau_{n},\mu_{m}\right)/h$ in the variables $(n,m)$ in
linear order. In this approximation appear two periods $T_{cl_{1}}$
and $T_{cl_{2}}$ of order $O(1)$ depending on the gradient of the
function $F$ at the point $\left(E_{1},E_{2}\right)$.

\subsection{Linear approximation and classical periods}
\begin{assumption}
Here, we suppose that $\frac{\partial F}{\partial X}\left(E_{1},E_{2}\right)\neq0,\,\frac{\partial F}{\partial Y}\left(E_{1},E_{2}\right)\neq0.$
\end{assumption}
\selectlanguage{english}%

\subsubsection{Semi-classical \foreignlanguage{english}{and classical periods }}
\selectlanguage{english}%
\begin{defn}
We define semi-classical periods $T_{scl_{1}}$and \foreignlanguage{english}{$T_{scl_{2}}$
by :\[
T_{scl_{1}}:=\frac{2\pi}{\frac{\partial F}{\partial X}\left(\tau_{n_{0}},\mu_{m_{0}}\right)\omega_{1}}\textrm{ and }T_{scl_{2}}:=\frac{2\pi}{\frac{\partial F}{\partial Y}\left(\tau_{n_{0}},\mu_{m_{0}}\right)\omega_{2}}.\]
}
\end{defn}
So, in linear order approximation, we have : 
\begin{prop}
Let $\alpha$ a real number such that $\alpha>1-2\min\delta_{i}$.
Then, uniformly for all $t\in\left[0,h^{\alpha}\right]$:\[
\mathbf{r}(t)=e^{-itF\left(\tau_{n_{0}},\mu_{m_{0}}\right)/h}{\displaystyle \sum_{n,m\in\mathbb{N}^{2}}\left|a_{n,m}\right|^{2}e^{-2i\pi t\left(\frac{n-n_{0}}{T_{scl_{1}}}+\frac{m-m_{0}}{T_{scl_{2}}}\right)}}+O\left(h^{\alpha+2\min\delta_{i}-1}\right).\]
\end{prop}
\begin{proof}
Let us introduce the difference $\varepsilon(t):=\varepsilon(t,h)$
defined by \[
\varepsilon(t):=\left|{\displaystyle \sum_{n,m\in\mathbb{N}^{2}}\left|a_{n,m}\right|^{2}e^{-i\frac{t}{h}F\left(\tau_{n},\mu_{m}\right)}}-e^{-itF\left(\tau_{n_{0}},\mu_{m_{0}}\right)/h}{\displaystyle \sum_{n,m\in\mathbb{N}^{2}}\left|a_{n,m}\right|^{2}e^{-2i\pi t\left(\frac{n-n_{0}}{T_{scl_{1}}}+\frac{m-m_{0}}{T_{scl_{2}}}\right)}}\right|.\]
For all integers $(n,m)\in\mathbb{N}^{2}$ the Taylor-Lagrange's formula
(at order 2) around $\left(\tau_{n_{0}},\mu_{m_{0}}\right)$ on the
function $F$ gives the existence of a real number $\theta=\theta\left(n,m,n_{0},m_{0}\right)\in\left]0,1\right[$
such that \[
F\left(\tau_{n},\mu_{m}\right)=F\left(\tau_{n_{0}},\mu_{m_{0}}\right)+\frac{2\pi h\left(n-n_{0}\right)}{T_{scl_{1}}}+\frac{2\pi h\left(m-m_{0}\right)}{T_{scl_{2}}}\]
\[
+\frac{1}{2}\frac{\partial^{2}F\left(\rho_{n,m}\right)}{\partial X^{2}}\omega_{1}^{2}h^{2}\left(n-n_{0}\right)^{2}+\frac{1}{2}\frac{\partial^{2}F\left(\rho_{n,m}\right)}{\partial Y^{2}}\omega_{2}^{2}h^{2}\left(m-m_{0}\right)^{2}\]
\[
+\frac{\partial^{2}F}{\partial X\partial Y}\left(\rho_{n,m}\right)\omega_{1}\omega_{2}h^{2}\left(n-n_{0}\right)\left(m-m_{0}\right),\]
with $\rho_{n,m}=\rho(n,m,n_{0},m_{0},h):=\left(\tau_{n_{0}}+\theta(\tau_{n}-\tau_{n_{0}}),\mu_{m_{0}}+\theta(\mu_{m}-\mu_{m_{0}})\right).$

So, we get\[
\varepsilon(t)=\left|\sum_{n,m\in\mathbb{N}^{2}}\left|a_{n,m}\right|^{2}e^{-2i\pi t\left(\frac{n-n_{0}}{T_{scl_{1}}}+\frac{m-m_{0}}{T_{scl_{2}}}\right)}\left[e^{-i2\pi tR_{n,m}(h)}-1\right]\right|\]
where we have used the notation\[
R_{n,m}(h):=\frac{h\mbox{\ensuremath{\omega}}_{1}^{2}(n-n_{0})^{2}}{4\mbox{\ensuremath{\pi}}}\frac{\partial^{2}F\left(\rho_{n,m}\right)}{\partial X^{2}}+\frac{h\mbox{\ensuremath{\omega}}_{2}^{2}(m-m_{0})^{2}}{4\mbox{\ensuremath{\pi}}}\frac{\partial^{2}F\left(\rho_{n,m}\right)}{\partial Y^{2}}\]
\[
+\frac{h\mbox{\ensuremath{\omega}}_{1}\omega_{2}(n-n_{0})(m-m_{0})}{2\pi}\frac{\partial^{2}F\left(\rho_{n,m}\right)}{\partial X\partial Y}.\]
With the sets $\Gamma,$$\Delta$ and by triangular inequality, we
obtain for all $t\geq0$\[
\varepsilon(t)\leq\left|\sum_{n,m\in\Delta}\left|a_{n,m}\right|^{2}e^{-2i\pi t\left(\frac{n-n_{0}}{T_{scl_{1}}}+\frac{m-m_{0}}{T_{scl_{2}}}\right)}\left[e^{-i2\pi tR_{n,m}(h)}-1\right]\right|+2\sum_{n,m\in\Gamma}\left|a_{n,m}\right|^{2}.\]
For all $t\geq0$, for $h$ small enough and for all integers $(n,m)\in\Delta,$
we observe that \[
\frac{th\mbox{\ensuremath{\omega}}_{1}^{2}(n-n_{0})^{2}}{4\mbox{\ensuremath{\pi}}}\frac{\partial^{2}F\left(\rho_{n,m}\right)}{\partial X^{2}}\leq tK_{1}h^{2\delta_{1}-1};\]
\[
\frac{th\mbox{\ensuremath{\omega}}_{2}^{2}(m-m_{0})^{2}}{4\mbox{\ensuremath{\pi}}}\frac{\partial^{2}F\left(\rho_{n,m}\right)}{\partial Y^{2}}\leq tK_{2}h^{2\delta_{2}-1};\]
\[
\frac{th\mbox{\ensuremath{\omega}}_{1}\omega_{2}(n-n_{0})(m-m_{0})}{2\pi}\frac{\partial^{2}F\left(\rho_{n,m}\right)}{\partial X\partial Y}\leq tK_{12}h^{\delta_{1}+\delta_{2}-1};\]
where $K_{1},K_{2},K_{12}>0$ are constants which does not depend
on $h$. Indeed : let us denotes by $B\left((E_{1},E_{2}),r\right)$
the Euclidian ball of dimension 2 with center $(E_{1},E_{2})$ and
radius $r$; since $\lim_{h\rightarrow0}\left(\tau_{n_{0}},\mu_{m_{0}}\right)=\left(E_{1},E_{2}\right)$
we obtain that $\forall\varepsilon>0,\,\exists h_{0}>0$, such that
for all $h\leq h_{0}$, $\left(\tau_{n_{0}},\mu_{m_{0}}\right)\in B\left((E_{1},E_{2}),\varepsilon\right)$;
next for all integers $(n,m)\in\Delta,$ we have $\left|\theta(\tau_{n}-\tau_{n_{0}})\right|=h\omega_{1}\theta\left|n-n_{0}\right|\leq\omega_{1}h^{\delta_{1}}$
and $\left|\theta(\mu_{m}-\mu_{n_{0}})\right|=h\omega_{2}\theta\left|m-m_{0}\right|\leq\omega_{2}h^{\delta_{2}}$,
this means that for $h$ small enough ($h\leq h_{0}$) we have \[
\rho_{n,m}\in B\left((E_{1},E_{2}),\varepsilon\right);\]
therefore we obtain for all $h\leq h_{0}$,\[
\left|\frac{\partial^{2}F}{\partial X^{2}}\left(\rho_{n,m}\right)\right|\leq\sup_{(x,y)\in B\left((E_{1},E_{2}),\varepsilon\right)}\left|\frac{\partial^{2}F}{\partial X^{2}}(x,y)\right|\]
and this quantity is $>0$ and does not depend on $h$. Next \foreignlanguage{english}{we
have for all }$t\in\left[0,h^{\alpha}\right]$ \[
t\left|R_{n,m}(h)\right|\leq K_{1}h^{\alpha+2\delta_{1}-1}+K_{2}h^{\alpha+2\delta_{2}-1}+K_{1,2}h^{\alpha+\delta_{1}+\delta_{2}-1}\]
\[
\leq Mh^{\alpha-1}\left(h^{2\delta_{1}}+h^{2\delta_{2}}+h^{\delta_{1}+\delta_{2}}\right)=3Mh^{2\min\delta_{i}+\alpha-1};\]
where $M:=\max\left(K_{1},K_{2},K_{12}\right);$ with (by hypothesis)
$2\min\delta_{i}+\alpha-1>0.$ This implies that for all $t\in\left[0,h^{\alpha}\right]$
and for all integers $(n,m)\in\Delta$ we get \[
e^{-i2\pi tR_{n,m}(h)}-1=O\left(h^{2\min\delta_{i}+\alpha-1}\right);\]
and consequently we have for \foreignlanguage{english}{all }$t\in\left[0,h^{\alpha}\right]$
\[
\left|\sum_{n,m\in\Delta}\left|a_{n,m}\right|^{2}e^{-2i\pi t\left(\frac{n-n_{0}}{T_{scl_{1}}}+\frac{m-m_{0}}{T_{scl_{2}}}\right)}\left[e^{-i2\pi tR_{n,m}(h)}-1\right]\right|\]
\[
=O\left(h^{2\min\delta_{i}+\alpha-1}\right)\sum_{n,m\in\mathbb{N}^{2}}\left|a_{n,m}\right|^{2}=O\left(h^{2\min\delta_{i}+\alpha-1}\right).\]
Finally, for all $t\in\left[0,h^{\alpha}\right]$ we have  $\varepsilon(t)=O\left(h^{2\min\delta_{i}+\alpha-1}\right).$ 
\end{proof}
The semi-classical periods $T_{scl_{i}}$ depend on the parameter
$h$. Later we consider two cases : $T_{scl_{1}}/T_{scl_{2}}\in\mathbb{Q}$
or not. Consequently we don't make commensurability hypothesis on
the number $T_{scl_{1}}/T_{scl_{2}}$ valid up for all $h>0$, so
we prefer introduce two quantities which does not depend on $h$ to
make latter commensurability hypothesis. So we replace semi-classical
periods $T_{scl_{i}}$ by semi-classical periods $T_{cl_{i}}$ .
\begin{defn}
We define classical periods $T_{cl_{1}}$and \foreignlanguage{english}{$T_{cl_{2}}$
by :\[
T_{cl_{1}}:=\frac{2\pi}{\frac{\partial F}{\partial X}\left(E_{1},E_{2}\right)\omega_{1}}\textrm{ and }T_{cl_{2}}:=\frac{2\pi}{\frac{\partial F}{\partial Y}\left(E_{1},E_{2}\right)\omega_{2}}.\]
}
\end{defn}
An obvious remark is that for all $j\in\left\{ 1,2\right\} $ we have
\foreignlanguage{english}{$\lim_{h\rightarrow0}T_{scl_{j}}=T_{cl_{j}}.$}
\begin{prop}
Let $\tau$ be a real number such that $\tau>-\min\delta_{i}$. Then,
uniformly for all $t\in\left[0,h^{\tau}\right]$: \[
{\displaystyle \sum_{n,m\in\mathbb{N}^{2}}\left|a_{n,m}\right|^{2}e^{-2i\pi t\left(\frac{n-n_{0}}{T_{scl_{1}}}+\frac{m-m_{0}}{T_{scl_{2}}}\right)}}=\sum_{n,m\in\mathbb{N}^{2}}\left|a_{n,m}\right|^{2}e^{-2i\pi t\left(\frac{n-n_{0}}{T_{cl_{1}}}+\frac{m-m_{0}}{T_{cl_{2}}}\right)}+O\left(h^{\tau+\min\delta_{i}}\right).\]
\end{prop}
\begin{proof}
We observe that\[
\left|\sum_{n,m\in\mathbb{N}^{2}}\left|a_{n,m}\right|^{2}\left[e^{-2i\pi t\left(\frac{n-n_{0}}{T_{scl_{1}}}+\frac{m-m_{0}}{T_{scl_{2}}}\right)}-e^{-2i\pi t\left(\frac{n-n_{0}}{T_{cl_{1}}}+\frac{m-m_{0}}{T_{cl_{2}}}\right)}\right]\right|\leq\sum_{n,m\in\Gamma}2\left|a_{n,m}\right|^{2}\]
\[
+2\sum_{n,m\in\Delta}\left|a_{n,m}\right|^{2}\left[\left|2\pi t(n-n_{0})\left(\frac{1}{T_{scl_{1}}}-\frac{1}{T_{cl_{1}}}\right)\right|+\left|2\pi t(m-m_{0})\left(\frac{1}{T_{scl_{2}}}-\frac{1}{T_{cl_{2}}}\right)\right|\right],\]
because $\left|e^{iX_{1}}e^{iX_{2}}-e^{iY_{1}}e^{iY_{2}}\right|\leq2\left|X_{1}-Y_{1}\right|+2\left|X_{2}-Y_{2}\right|.$ 

Next for all $t\geq0$ we have \[
\left|2\pi t(n-n_{0})\left(\frac{1}{T_{scl_{1}}}-\frac{1}{T_{cl_{1}}}\right)\right|=\left|2\pi t(n-n_{0})\left(\frac{T_{cl_{1}}-T_{scl_{1}}}{T_{scl_{1}}T_{cl_{1}}}\right)\right|,\]
and we know that \[
T_{cl_{1}}-T_{scl_{1}}=\frac{2\pi}{\omega_{1}}\frac{\frac{\partial F}{\partial X}\left(\tau_{n_{0}},\mu_{m_{0}}\right)-\frac{\partial F}{\partial X}\left(E_{1},E_{2}\right)}{\frac{\partial F}{\partial X}\left(E_{1},E_{2}\right)\frac{\partial F}{\partial X}\left(\tau_{n_{0}},\mu_{m_{0}}\right)}:\]
first, applying the inequality of Lagrange we obtain :\[
\left|\frac{\partial F}{\partial X}\left(\tau_{n_{0}},\mu_{m_{0}}\right)-\frac{\partial F}{\partial X}\left(E_{1},E_{2}\right)\right|\]
\[
\leq\sup_{x,y\in B\left((E_{1},E_{2}),1\right)}\left\Vert \nabla\left(\frac{\partial F}{\partial X}\right)(x,y)\right\Vert _{\mathbb{R}^{2}}\left\Vert \left(\tau_{n_{0}},\mu_{m_{0}}\right)-\left(E_{1},E_{2}\right)\right\Vert _{\mathbb{R}^{2}}\]
\[
\leq M\sqrt{\left(\tau_{n_{0}}-E_{1}\right)^{2}+\left(\mu_{m_{0}}-E_{2}\right)^{2}}\leq Mh\frac{\sqrt{2}}{2}.\]
where $M>0$ and does not depend on $h$. 

On the other hand, since we suppose $\frac{\partial F}{\partial X}\left(E_{1},E_{2}\right)\neq0$,
there exists $\varepsilon_{1}>0$ and $r_{1}>0$ such that for all
$(x,y)\in B\left((E_{1},E_{2}),r_{1}\right)$ we get \foreignlanguage{english}{\[
\left|\frac{\partial F}{\partial X}\left(x,y\right)\right|\geq\varepsilon_{1}.\]
We have seen that hence that there exists $h_{1}>0$ such that for
all $h\in\left]0,h_{1}\right[$ }\[
(\tau_{n_{0}},\mu_{m_{0}})\in B\left((E_{1},E_{2}),r_{1}\right);\]
as a consequence the application $h\mapsto\frac{1}{\frac{\partial F}{\partial X}\left(E_{1},E_{2}\right)\frac{\partial F}{\partial X}\left(\tau_{n_{0}},\mu_{m_{0}}\right)}$
is bounded on the open set \foreignlanguage{english}{$\left]0,h_{1}\right[$;
indeed for all $h\in\left]0,h_{1}\right[$\[
\left|\frac{1}{\frac{\partial F}{\partial X}\left(E_{1},E_{2}\right)\frac{\partial F}{\partial X}\left(\tau_{n_{0}},\mu_{m_{0}}\right)}\right|\leq\frac{1}{\varepsilon_{1}^{2}}<+\infty\]
hence, with $M^{\prime}:=\frac{2\pi}{\omega_{1}}Mh\frac{\sqrt{2}}{2}\frac{1}{\varepsilon_{1}^{2}}$
, for all $h\in\left]0,h_{1}\right[$ we have $\left|T_{cl_{1}}-T_{scl_{1}}\right|\leq hM^{\prime}.$
Next, since\[
\left|\frac{1}{T_{scl_{1}}T_{cl_{1}}}\right|\leq\frac{\omega_{1}\omega_{2}}{4\pi^{2}}\left|\frac{\partial F}{\partial X}\left(E_{1},E_{2}\right)\frac{\partial F}{\partial X}\left(\tau_{n_{0}},\mu_{m_{0}}\right)\right|\]
\[
\leq\frac{\omega_{1}\omega_{2}}{4\pi^{2}}\left(\sup_{x,y\in B\left((E_{1},E_{2}),1\right)}\left|\frac{\partial F}{\partial X}(x,y)\right|\right)^{2}<\infty\]
there exists a constant $C_{1}>0$ which does not depend on $h$ such
that for all $h\in\left]0,h_{1}\right[$ we get $\left|1/T_{scl_{1}}-1/T_{cl_{1}}\right|\leq C_{1}h.$
In a similar way there exists $C_{2}>0$ and $h_{2}>0$ such that
for all $h\in\left]0,h_{2}\right[$ we get $\left|1/T_{scl_{2}}-1/T_{cl_{2}}\right|\leq C_{2}h.$
As a consequence, for all $h\in\left]0,h^{*}\right[$ where $h^{*}:=\min h_{i}$,
for all $t\in\left[0,h^{\tau}\right]$ with $\tau\in\mathbb{R}$,
and for all integers $(n,m)\in\triangle$ we have : }\[
\left|t(n-n_{0})\left(\frac{1}{T_{scl_{1}}}-\frac{1}{T_{cl_{1}}}\right)\right|\leq C_{1}h^{\upsilon+\delta_{1}},\:\left|t(m-m_{0})\left(\frac{1}{T_{scl_{2}}}-\frac{1}{T_{cl_{2}}}\right)\right|\leq C_{2}h^{\upsilon+\delta_{2}};\]
we thus obtain for all $t,(n,m)\in\left[0,h^{\upsilon}\right]\times\Delta$\[
\left|t(n-n_{0})\left(\frac{1}{T_{scl_{1}}}-\frac{1}{T_{cl_{1}}}\right)+t(m-m_{0})\left(\frac{1}{T_{scl_{2}}}-\frac{1}{T_{cl_{2}}}\right)\right|\leq Mh^{\tau+\min\delta_{i}}.\]
Therefore\[
\left|\sum_{n,m\in\mathbb{N}^{2}}\left|a_{n,m}\right|^{2}\left[e^{-2i\pi t\left(\frac{n-n_{0}}{T_{scl_{1}}}+\frac{m-m_{0}}{T_{scl_{2}}}\right)}-e^{-2i\pi t\left(\frac{n-n_{0}}{T_{cl_{1}}}+\frac{m-m_{0}}{T_{cl_{2}}}\right)}\right]\right|\]
\[
\leq4\pi Mh^{\tau+\min\delta_{i}}\sum_{n,m\in\Delta}\left|a_{n,m}\right|^{2}+O\left(h^{\infty}\right)=O\left(h^{\tau+\min\delta_{i}}\right).\]

\end{proof}

\subsubsection{Comparison between classical periods and the time scale $\left[0,h^{\alpha}\right]$}

In proposition 3.3 the hypothesis on $\alpha$ is that $\alpha>1-2\min\delta_{i}$;
therefore with $\delta_{i}\in]\frac{1}{2},1[$ we can make a {}``good
choice'' for $\alpha;$ i.e. to have $\alpha<0$. Hence for $h$
small enough we obtain :\[
\left[0,T_{cl_{i}}\right]\subset\left[0,h^{\alpha}\right].\]
Next, since $-\min\delta_{i}-\left(1-2\min\delta_{i}\right)=-1+\min\delta_{i}\leq0$
we get\[
h^{-\min\delta_{i}}\gg h^{1-2\min\delta_{i}};\]
this means that we can choose to take $\tau=\alpha.$

\subsubsection{The linear approximation $\mathbf{a}_{1}$}

In conclusion, the linear approximation of the autocorrelation function
on the time scale $\left[0,h^{\alpha}\right]$ is :
\begin{defn}
The linear approximation of the autocorrelation function is\[
\mathbf{\mathbf{a_{1}}\,:\,}t\mapsto{\displaystyle \sum_{n,m\in\mathbb{N}^{2}}\left|a_{n,m}\right|^{2}e^{-2i\pi t\left(\frac{n-n_{0}}{T_{cl_{1}}}+\frac{m-m_{0}}{T_{cl_{2}}}\right)}.}\]

\end{defn}

\subsection{Geometrical interpretation of classical periods}

The periods $T_{cl_{i}}$ have geometrical interpretation. For $E_{1},E_{2}>0$
consider the energy level set $M_{E_{1},E_{2}}:=p_{1}^{-1}\left(E_{1}\right)\cap p_{2}^{-1}\left(E_{2}\right)\subset\mathbb{R}^{4}$
, this manifold is isomorphic to the torus $\sqrt{\frac{2E_{1}}{\omega_{1}}}\mathbb{S}^{1}\times\sqrt{\frac{2E_{2}}{\omega_{2}}}\mathbb{S}^{1},$
here $\mathbb{S}^{1}$ is the one-dimension circle. Start with the
calculus of the Hamiltonian flow of $p=F(p_{1},p_{2})$ with an initial
point $m_{0}\in M_{E_{1},E_{2}}.$ So the Hamilton's equations are\[
\left(\begin{array}{c}
\dot{x_{1}}(t)\\
\dot{x_{2}}(t)\\
\dot{\xi_{1}}(t)\\
\dot{\xi_{2}}(t)\end{array}\right)=\left(\begin{array}{c}
a\xi_{1}(t)\\
b\xi_{2}(t)\\
-ax_{1}(t)\\
-bx_{2}(t)\end{array}\right)\]
where we have used the notation $a:=\frac{\partial F}{\partial X}\left(E_{1},E_{2}\right)\omega_{1},\; b:=\frac{\partial F}{\partial Y}\left(E_{1},E_{2}\right)\omega_{2}.$
For all $j\in\left\{ 1,2\right\} $, let us consider the complex number
$Z_{j}(t):=x_{j}(t)+i\xi_{j}(t)$; from the Hamilton equations we
obtain the equalities $\dot{Z_{1}(t)}=-iaZ_{1}(t),\;\dot{Z_{2}(t)}=-ibZ_{2}(t).$
Therefore we get \[
Z_{1}(t)=Z_{1}(0)e^{-iat},\; Z_{2}(t)=Z_{2}(0)e^{-ibt}\]
and \[
|Z_{1}(0)|^{2}=x_{1}^{2}(0)+\xi_{1}^{2}(0)=\frac{2E_{1}}{a},\,|Z_{2}(0)|^{2}=x_{2}^{2}(0)+\xi_{2}^{2}(0)=\frac{2E_{2}}{b};\]
this means that the Hamiltonian's flow in complex coordinate is given
by 

\[
\varphi_{t}:\,\left(\begin{array}{c}
Z_{1}(0)\\
Z_{2}(0)\end{array}\right)\mapsto\left(\begin{array}{c}
Z_{1}(t)\\
Z_{2}(t)\end{array}\right).\]
In angular coordinate the flow is given by\[
\varphi_{t}:\,\left(\begin{array}{c}
\theta_{1,0}\\
\theta_{2,0}\end{array}\right)\mapsto\left(\begin{array}{c}
\theta_{1,0}-t\frac{a}{2\pi}\\
\theta_{2,0}-t\frac{b}{2\pi}\end{array}\right)\]
with $\theta_{j,0}\equiv\frac{\arg Z_{j}(0)}{2\pi}\;\;\left[1\right].$
So we have exactly the classical periods of the Hamiltonian's flow
: \[
\frac{2\pi}{a}=\frac{2\pi}{\frac{\partial F}{\partial X}\left(E_{1},E_{2}\right)\omega_{1}}=T_{cl_{1}},\;\frac{2\pi}{b}=\frac{2\pi}{\frac{\partial F}{\partial Y}\left(E_{1},E_{2}\right)\omega_{2}}=T_{cl_{2}}.\]
It's well know that if the periods are commensurate the flow is periodic
on the torus. In opposite the flow is quasi-periodic on the torus.

\vspace{+0.25cm}

\subsection{The principal part of the function $\mathbf{a_{1}}$}

Now, let us study in details the function $\mathbf{a_{1}}(t)$ on
the time scale \textit{$\left[0,\max T_{cl_{i}}\right]$. }Start by
a technical proposition :
\begin{prop}
For all $t\geq0$ we have \[
\sum_{n,m\in\mathbb{\mathbb{Z}}^{2}}\left|a_{n,m}\right|^{2}e^{-2i\pi t\left(\frac{n-n_{0}}{T_{cl_{1}}}+\frac{m-m_{0}}{T_{cl_{2}}}\right)}\]
\[
=\frac{1}{\mathfrak{F}\left(\chi^{2}\right)(0,0)}\sum_{\ell,s\in\mathbb{Z}^{2}}\mathfrak{F}\left(\chi^{2}\right)\left(-\frac{h^{\delta_{1}^{\prime}-1}}{\omega_{1}}\left(\ell+\frac{t}{T_{cl_{1}}}\right),-\frac{h^{\delta_{2}^{\prime}-1}}{\omega_{2}}\left(s+\frac{t}{T_{cl_{2}}}\right)\right).\]
\end{prop}
\begin{proof}
The trick here is just to use the Poisson formula, so let us consider
the function $\Omega_{t}$ defined by \[
\Omega_{t}:\left\{ \begin{array}{cc}
\mathbb{R}^{2}\rightarrow\mathbb{\mathbb{C}}\\
\\(x_{1},x_{2})\mapsto\left|a_{x_{1},x_{2}}\right|^{2}e^{-2i\pi t\frac{(x_{1}-n_{0})}{T_{cl_{1}}}}e^{-2i\pi t\frac{(x_{2}-m_{0})}{T_{cl_{2}}}}\end{array}\right.\]
where $t\in\mathbb{R}$ is a parameter. For all integers $(n,m)\in\mathbb{\mathbb{Z}}^{2}$
we have \[
\left|a_{n,m}\right|^{2}e^{-2i\pi t\left(\frac{n-n_{0}}{T_{cl_{1}}}+\frac{m-m_{0}}{T_{cl_{2}}}\right)}=\Omega_{t}(n,m).\]
So clearly, the function $\Omega_{t}\in\mathcal{S}(\mathbb{R}^{2})$,
then the Fourier transform $\mathfrak{F}\left(\Omega_{t}\right)$
is equal, for all $\zeta_{1},\zeta_{2}\in\mathbb{R}^{2}$\[
\mathfrak{F}\left(\Omega_{t}\right)\left(\zeta_{1},\zeta_{2}\right)=\int_{-\infty}^{+\infty}\int_{-\infty}^{+\infty}\Omega_{t}\left(x_{1},x_{2}\right)e^{-2i\pi x_{1}\zeta_{1}}e^{-2i\pi x_{2}\zeta_{2}}\, dx_{1}dx_{2};\]
therefore for all $\zeta_{1},\zeta_{2}\in\mathbb{R}^{2}$ we get\[
\mathfrak{F}\left(\Omega_{t}\right)\left(\zeta_{1},\zeta_{2}\right)=\frac{e^{-2i\pi\left(n_{0}\zeta_{1}+m_{0}\zeta_{2}\right)}}{\mathfrak{F}\left(\chi^{2}\right)(0,0)}\mathfrak{F}\left(\chi^{2}\right)\left(-\frac{h^{\delta_{1}^{\prime}-1}}{\omega_{1}}\left(\zeta_{1}+\frac{t}{T_{cl_{1}}}\right),-\frac{h^{\delta_{2}^{\prime}-1}}{\omega_{2}}\left(\zeta_{2}+\frac{t}{T_{cl_{2}}}\right)\right).\]
It comes from the Poisson formula the equality \[
\sum_{n,m\in\mathbb{Z}^{2}}\Omega_{t}(n,m)=\sum_{\ell,s\in\mathbb{Z}^{2}}\mathfrak{F}\left(\Omega_{t}\right)(\ell,s)\]
\textit{\[
=\frac{1}{\mathfrak{F}\left(\chi^{2}\right)(0,0)}\sum_{\ell,s\in\mathbb{Z}^{2}}\mathfrak{F}\left(\chi^{2}\right)\left(-\frac{h^{\delta_{1}^{\prime}-1}}{\omega_{1}}\left(\ell+\frac{t}{T_{cl_{1}}}\right),-\frac{h^{\delta_{2}^{\prime}-1}}{\omega_{2}}\left(s+\frac{t}{T_{cl_{2}}}\right)\right)\]
}which gives the proposition.
\end{proof}
Since the function $\mathfrak{F}\left(\chi^{2}\right)\in\mathcal{S}(\mathbb{R}^{2})$,
we observe that only index $\ell,s\in\mathbb{Z}^{2}$ such that $\ell+\frac{t}{T_{cl_{1}}}$
or $s+\frac{t}{T_{cl_{2}}}$ are close to zero are important in the
sum. More precisely : 
\begin{defn}
For all $t\geq0$, let us define the integers $\ell_{i}(t)=\ell_{i}(t,h,E)$
as the closest integers to the real numbers $-t/T_{cl_{i}}$; i.e
:\[
\ell_{i}(t)+\frac{t}{T_{cl_{i}}}=d\left(t,T_{cl_{i}}\mathbb{Z}\right);\]
where $d(.,.)$ denote the Euclidiean distance on $\mathbb{R}$.\end{defn}
\begin{rem}
Without loss of generality, we may suppose the integers $\ell_{i}(t)$
are unique. On the other hand, for all integer $\ell\in\mathbb{Z}$
such that $\ell\neq\ell_{i}(t)$ we get :\[
\left|\ell+\frac{t}{T_{cl_{i}}}\right|\geq\frac{1}{2}.\]
\end{rem}
\begin{lem}
Uniformly for $t\geq0$ we have :\[
\mathbf{a_{1}}(t)=\frac{1}{\mathfrak{F}\left(\chi^{2}\right)(0,0)}\mathfrak{F}\left(\chi^{2}\right)\left(-\frac{h^{\delta_{1}^{\prime}-1}}{\omega_{1}}d\left(T_{cl_{1}}\mathbb{Z},t\right),-\frac{h^{\delta_{2}^{\prime}-1}}{\omega_{2}}d\left(T_{cl_{2}}\mathbb{Z},t\right)\right)+O\left(h^{\infty}\right).\]
\end{lem}
\begin{proof}
Since $\mathfrak{F}\left(\chi^{2}\right)\in\mathcal{S}(\mathbb{R}^{2})$
we have\[
\forall k,d\in\mathbb{N}^{*2},\,\exists B_{k,d}>0,\,\forall\zeta_{1},\zeta_{2}\in\mathbb{R}^{2},\,\left|\mathfrak{F}\left(\chi^{2}\right)(\zeta_{1},\zeta_{2})\right|\leq\frac{B_{k,d}}{\left(1+\left|\zeta_{1}\right|\right)^{k}\left(1+\left|\zeta_{2}\right|\right)^{d}}.\]
Next, it then follow from the proposition above and from the lemma
2.4 that for all $t\geq0$ \textit{\[
\mathbf{a_{1}}(t)=\frac{1}{\mathfrak{F}\left(\chi^{2}\right)(0)}\sum_{\ell,s\in\mathbb{Z}^{2}}\mathfrak{F}\left(\chi^{2}\right)\left(-\frac{h^{\delta_{1}^{\prime}-1}}{\omega_{1}}\left(\ell+\frac{t}{T_{cl_{1}}}\right),-\frac{h^{\delta_{2}^{\prime}-1}}{\omega_{2}}\left(s+\frac{t}{T_{cl_{2}}}\right)\right)\]
\[
=\frac{1}{\mathfrak{F}\left(\chi^{2}\right)(0,0)}\mathfrak{F}\left(\chi^{2}\right)\left(-\frac{h^{\delta_{1}^{\prime}-1}}{\omega_{1}}d\left(t,T_{cl_{1}}\mathbb{Z}\right),-\frac{h^{\delta_{2}^{\prime}-1}}{\omega_{2}}d\left(t,T_{cl_{2}}\mathbb{Z}\right)\right)+O\left(h^{\infty}\right).\]
}Next, for all $t\geq0$ \[
\mathbf{a_{1}}(t)=\sum_{n,m\in\mathbb{\mathbb{\mathbb{\mathbb{Z}}}}^{2}}\left|a_{n,m}\right|^{2}e^{-2i\pi t\left(\frac{n-n_{0}}{T_{cl_{1}}}+\frac{m-m_{0}}{T_{cl_{2}}}\right)}-\sum_{n,m\in\mathbb{\mathbb{\mathbb{\mathbb{Z}}}}^{2}-\mathbb{\mathbb{\mathbb{\mathbb{N}}}}^{2}}\left|a_{n,m}\right|^{2}e^{-2i\pi t\left(\frac{n-n_{0}}{T_{cl_{1}}}+\frac{m-m_{0}}{T_{cl_{2}}}\right)};\]
thus\[
\left|\mathbf{a_{1}}(t)-\frac{1}{\mathfrak{F}\left(\chi^{2}\right)(0,0)}\mathfrak{F}\left(\chi^{2}\right)\left(-\frac{h^{\delta_{1}^{\prime}-1}}{\omega_{1}}d\left(T_{cl_{1}}\mathbb{Z},t\right),-\frac{h^{\delta_{2}^{\prime}-1}}{\omega_{2}}d\left(T_{cl_{2}}\mathbb{Z},t\right)\right)\right|\]
\[
\leq\sum_{n,m\in\mathbb{\mathbb{\mathbb{\mathbb{Z}}}}^{2}-\mathbb{\mathbb{\mathbb{\mathbb{N}}}}^{2}}\left|a_{n,m}\right|^{2}+O\left(h^{\infty}\right).\]
For finish, we observe \[
\sum_{n,m\in\mathbb{\mathbb{\mathbb{\mathbb{Z}}}}^{2}-\mathbb{\mathbb{\mathbb{\mathbb{N}}}}^{2}}\left|a_{n,m}\right|^{2}\]
\foreignlanguage{english}{\[
=\sum_{n=0}^{+\infty}\sum_{m=-\infty}^{-1}\left|a_{n,m}\right|^{2}+\sum_{n=-\infty}^{-1}\sum_{m=-\infty}^{-1}\left|a_{n,m}\right|^{2}+\sum_{n=-\infty}^{-1}\sum_{m=0}^{+\infty}\left|a_{n,m}\right|^{2};\]
}and an obvious consequence of the lemma 2.4 is that $\sum_{n,m\in\mathbb{\mathbb{\mathbb{\mathbb{Z}}}}^{2}-\mathbb{\mathbb{\mathbb{\mathbb{N}}}}^{2}}\left|a_{n,m}\right|^{2}=O\left(h^{\infty}\right).$ 
\end{proof}

\subsection{Behaviour of the function $\mathbf{a_{1}}$ : case $\frac{T_{cl_{1}}}{T_{cl_{2}}}\in\mathbb{Q}$ }

In this subsection we suppose $\frac{T_{cl_{1}}}{T_{cl_{2}}}=\frac{b}{a}\in\mathbb{Q};$
hence $aT_{cl_{1}}=bT_{cl_{2}}.$
\begin{defn}
If the classical periods $T_{cl_{1}},T_{cl_{2}}$ are commensurate
the classical period of the global system is defined by $T_{cl}:=aT_{cl_{1}}=bT_{cl_{2}}.$

Now, we can formulate an important result of this section :\end{defn}
\begin{thm}
We have :

\textbf{(i)} for $t$ real such that $t\in T_{cl}\mathbb{Z}$ we get
(i.e. for all $i\in\{1,2\}$, $d\left(t,T_{cl_{i}}\mathbb{Z}\right)=0)$
\[
\mathbf{a_{1}}(t)=1.\]
\textbf{(ii)} If there exists $i\in\{1,2\}$ such that $d\left(T_{cl_{i}}\mathbb{Z},t\right)>h^{1-\delta_{i}^{\prime}}$
then :\[
\mathbf{a_{1}}(t)=O(h^{\infty}).\]
\end{thm}
\begin{proof}
The first point (i) is clear. For the second : it follows from the
lemma 3.10 that \textit{\[
\mathbf{a_{1}}(t)=\frac{1}{\mathfrak{F}\left(\chi^{2}\right)(0,0)}\mathfrak{F}\left(\chi^{2}\right)\left(-\frac{h^{\delta_{1}^{\prime}-1}}{\omega_{1}}d\left(T_{cl_{1}}\mathbb{Z},t\right),-\frac{h^{\delta_{2}^{\prime}-1}}{\omega_{2}}d\left(T_{cl_{2}}\mathbb{Z},t\right)\right)+O\left(h^{\infty}\right);\]
}since the function $\mathfrak{F}\left(\chi^{2}\right)\in\mathcal{S}(\mathbb{R}^{2})$
we have \[
\forall q\in\mathbb{N},\,\exists D_{q}>0,\,\forall\zeta_{1},\zeta_{2}\in\mathbb{R}^{2},\,\left|\mathfrak{F}\left(\chi^{2}\right)\left(\zeta_{1},\zeta_{2}\right)\right|\leq\frac{D_{q}}{\left(1+\left|\zeta_{1}\right|+\left|\zeta_{2}\right|\right)^{q}}\]
and therefore\textit{\[
\left|\mathfrak{F}\left(\chi^{2}\right)\left(-\frac{h^{\delta_{1}^{\prime}-1}}{\omega_{1}}d\left(T_{cl_{1}}\mathbb{Z},t\right),-\frac{h^{\delta_{2}^{\prime}-1}}{\omega_{2}}d\left(T_{cl_{2}}\mathbb{Z},t\right)\right)\right|\]
}\[
\leq\frac{D_{q}}{\left(1+\frac{h^{\delta_{1}^{\prime}-1}}{\omega_{1}}d\left(T_{cl_{1}}\mathbb{Z},t\right)+\frac{h^{\delta_{2}^{\prime}-1}}{\omega_{2}}d\left(T_{cl_{2}}\mathbb{Z},t\right)\right)^{q}}.\]
Thus, if there exists $i\in\{1,2\}$ such that $d\left(T_{cl_{i}}\mathbb{Z},t\right)>\omega_{i}h^{1-\delta_{i}^{\prime}}$
then there exists $\varepsilon>0$ such that $d\left(T_{cl_{i}}\mathbb{Z},t\right)\geq\omega_{i}h^{1-\delta_{i}^{\prime}-\varepsilon}$
and thus for all $q\in\mathbb{N}$ we obtain\textit{\[
\left|\mathfrak{F}\left(\chi^{2}\right)\left(-\frac{h^{\delta_{1}^{\prime}-1}}{\omega_{1}}d\left(T_{cl_{1}}\mathbb{Z},t\right),-\frac{h^{\delta_{2}^{\prime}-1}}{\omega_{2}}d\left(T_{cl_{2}}\mathbb{Z},t\right)\right)\right|\leq D_{q}h^{\varepsilon q};\]
}hence we get\textit{\[
\mathfrak{F}\left(\chi^{2}\right)\left(-\frac{h^{\delta_{1}^{\prime}-1}}{\omega_{1}}d\left(T_{cl_{1}}\mathbb{Z},t\right),-\frac{h^{\delta_{2}^{\prime}-1}}{\omega_{2}}d\left(T_{cl_{2}}\mathbb{Z},t\right)\right)=O(h^{\infty}).\]
}
\end{proof}

\subsection{Behaviour of the function $\mathbf{a_{1}}$ : case  $\frac{T_{cl_{1}}}{T_{cl_{2}}}\notin\mathbb{Q}$ }

Let us now tackle an important case : the case $\frac{b}{a}=\frac{T_{cl_{1}}}{T_{cl_{2}}}$
is not a fraction of $\mathbb{Q}$. Here, there not exists classical
common period, the Hamiltonian flow is not periodic on the torus. 

First, we note that, in view of lemma 3.10, the behaviour of the function
$\mathbf{a}_{\mathbf{1}}$ is given by the function : \[
t\mapsto\frac{1}{\mathfrak{F}\left(\chi^{2}\right)(0,0)}\mathfrak{F}\left(\chi^{2}\right)\left(-\frac{h^{\delta_{1}^{\prime}-1}}{\omega_{1}}d\left(T_{cl_{1}}\mathbb{Z},t\right),-\frac{h^{\delta_{2}^{\prime}-1}}{\omega_{2}}d\left(T_{cl_{2}}\mathbb{Z},t\right)\right).\]
Therefore, since the function $\mathfrak{F}\left(\chi^{2}\right)$
belongs to the space $\mathcal{S}(\mathbb{R}^{2}),$ we need to explain
simultaneously the evolutions of the distances $d\left(T_{cl_{j}}\mathbb{Z},t\right)$
depending on time. In another formulation, we want to analyze the
behaviour of the Euclidian distance between the segment line $\left(OM_{t}\right)$
where $O:=(0,0)$, $M_{t}:=(t,t)$ and the lattice \foreignlanguage{english}{$T_{cl_{1}}\mathbb{Z}\times T_{cl_{2}}\mathbb{Z}$}
depending on the time $t$ and the number $\frac{T_{cl_{1}}}{T_{cl_{2}}}=\frac{b}{a}$.
Precisely, we want to compare the distance $d\left(\left(OM_{t}\right),T_{cl_{1}}\mathbb{Z}\times T_{cl_{2}}\mathbb{Z}\right)$
with the real number $h^{\delta_{j}^{\prime}-1}$. For example, if
for a time $t^{*}$ the distance is larger than $h^{1-\delta_{j}^{\prime}}$
we get $\mathbf{a}_{\mathbf{1}}\left(t^{*}\right)=O(h^{\infty})$.

\vspace{+0.25cm}

Start this new subsection by some geometrical results and latter we
explain the study of the autocorrelation function $\mathbf{a_{1}}(t)$.

\subsubsection{Some general points}

In angular coordinates the Hamiltonian flow is :

\[
\varphi_{t}:\,\left\{ \begin{array}{cc}
\left[0,1\right]^{2}\rightarrow\left[0,1\right]^{2}\\
\\\left(\begin{array}{c}
\theta_{1,0}\\
\theta_{2,0}\end{array}\right)\mapsto\left(\begin{array}{c}
-\frac{at}{2\pi}+\theta_{1,0}\\
-\frac{bt}{2\pi}+\theta_{2,0}\end{array}\right) & .\end{array}\right.\]
Without loss of generality we may suppose that the initial data is
$\left(\begin{array}{c}
\theta_{1,0}\\
\theta_{2,0}\end{array}\right)=\left(\begin{array}{c}
0\\
0\end{array}\right)$ and that $a,b<0$. Therefore the Hamiltonian flow is given by $\varphi_{t}=\left(\begin{array}{c}
\mathbf{a}t\\
\mathbf{b}t\end{array}\right),$ where we have used the notation $\mathbf{a}:=-\frac{a}{2\pi}>0$
and $\mathbf{b}:=-\frac{b}{2\pi}>0.$ Recall that $T_{cl_{1}}=\left|\frac{2\pi}{a}\right|=\left|\frac{1}{\mathbf{a}}\right|$
and $T_{cl_{2}}=\left|\frac{2\pi}{b}\right|=\left|\frac{1}{\mathbf{b}}\right|.$
So, to understand the behaviour of the function 

\[
t\mapsto\frac{1}{\mathfrak{F}\left(\chi^{2}\right)(0,0)}\mathfrak{F}\left(\chi^{2}\right)\left(-\frac{h^{\delta_{1}^{\prime}-1}}{\omega_{1}}d\left(T_{cl_{1}}\mathbb{Z},t\right),-\frac{h^{\delta_{2}^{\prime}-1}}{\omega_{2}}d\left(T_{cl_{2}}\mathbb{Z},t\right)\right)\]
we need to explain the evolution of the Euclidian distance $d\left(\varphi_{t},\mathbb{Z}_{*}^{2}\right)$
depending on time $t$ and on the real number $\frac{\mathbf{b}}{\mathbf{a}}$.

\subsubsection{Suppose $\frac{b}{a}$ verify diophantin condition }

J. Liouville proved in 1884 the following theorem :
\begin{thm}
\textbf{(Liouville). }For all algebraic irrational number $\theta$
with degree $d\geq2$ there exists a constant $C=C(\theta)>0$ such
that the inequality\[
\left|\theta-\frac{p}{q}\right|\geq\frac{C}{q^{d}}\]
holds for all rationals $\frac{p}{q}$.
\end{thm}
In other words, algebraic numbers are bad approximation by rationals.
Finally, in 1955 K. F. Roth has considerably improved this result
(he was awarded the Field medal in 1958).
\begin{defn}
We say an irrational number $\theta$ satisfy a $\varepsilon$-diophantin
condition $(\varepsilon\geq0)$ if and only if there exists a constant
$C_{\varepsilon}>0$ such that\[
\left|\theta-\frac{p}{q}\right|\geq\frac{C_{\varepsilon}}{q^{2+\varepsilon}}.\]
holds for all $(p,q)\in\mathbb{Z}\times\mathbb{N}^{*}$. We denote
by $\mathcal{C}_{\varepsilon}$ the set of irrationals $\theta$ that
holds $\varepsilon$-diophantin condition. We say that an $\theta$
irrational $\theta$ is a Roth number if and only if $\theta\in{\displaystyle \bigcap_{\varepsilon>0}}\mathcal{C}_{\varepsilon}$;
i.e \[
\forall{\displaystyle \varepsilon>0},\,\exists C_{\varepsilon}>0;\,\forall(p,q)\in\mathbb{Z}\times\mathbb{N}^{*};\,\left|\theta-\frac{p}{q}\right|\geq\frac{C_{\varepsilon}}{q^{2+\varepsilon}}.\]

\end{defn}
There is a lot of Roth numbers examples : 
\begin{thm}
\textbf{(Thue-Siegel-Roth). }Every real algebraic irrational number
of degree $d\geq2$ is a Roth number.
\end{thm}
We have also the (see for example\textbf{ {[}Cas{]}}) :
\begin{thm}
The Lebesgue measure of Roth's numbers is infinite.\end{thm}
\begin{rem}
Let $\theta$ a $\varepsilon$-diophantin number. Since for all $p$
we have$\left|\theta-\frac{p}{1}\right|\geq C_{\varepsilon}$ and
$\left|\theta-p\right|\leq\frac{1}{2}\leq\frac{\sqrt{2}}{2}$; thus
we obtain that $0<C_{\varepsilon}\leq\frac{1}{2}$ holds for all $\varepsilon>0$
\end{rem}
Now, we estimate the Euclidian distance between the set $\left\{ \varphi_{t}\right\} _{t\in[0,T]}$
and $\mathbb{Z}_{*}^{2}:=\mathbb{Z}^{2}-\{(0,0)\}$ .
\begin{notation}
Let us denote by $\Delta$ and $\Gamma$ the following orthogonal
lines \[
\Delta:=\textrm{Vect}(\mathbf{a}e_{1}+\mathbf{\mathbf{b}}e_{2}),\;\Gamma:=\textrm{Vect}(-\mathbf{\mathbf{b}}e_{1}+\mathbf{a}e_{2})\]
where $(e_{1},e_{2})$ is the canonical basis of the vector space
$\mathbb{R}^{2}.$ Let us also considers $\pi_{\Delta}$ the orthogonal
projector on the line $\Delta$ and $\pi_{\Gamma}$ the orthogonal
projector on the line \foreignlanguage{english}{$\Gamma$.}
\end{notation}
\begin{center}
\textbf{\textit{Here we suppose $\theta=\frac{b}{a}$ is a Roth number.}}
\par\end{center}
\begin{lem}
For all $\varepsilon>0$ there exists $0<K_{\varepsilon}\leq C_{\varepsilon}$,
here $C_{\varepsilon}$ denotes the Roth constant of $\theta=\frac{b}{a}$,
such that \[
\left\Vert \pi_{\Delta}\left(ne_{1}+me_{2}\right)\right\Vert _{\mathbb{R}^{2}}\geq\frac{K_{\varepsilon}}{\left\Vert ne_{1}+me_{2}\right\Vert _{\mathbb{R}^{2}}^{1+\varepsilon}};\;\left\Vert \pi_{\Gamma}\left(ne_{1}+me_{2}\right)\right\Vert _{\mathbb{R}^{2}}\geq\frac{K_{\varepsilon}}{\left\Vert ne_{1}+me_{2}\right\Vert _{\mathbb{R}^{2}}^{1+\varepsilon}}\]
holds for all $(n,m)\in\mathbb{Z}_{*}^{2}$.\end{lem}
\begin{proof}
Let us denotes by $u=\left(\begin{array}{c}
u_{1}\\
u_{2}\end{array}\right):=\frac{1}{\sqrt{\mathbf{a}^{2}+\mathbf{b}^{2}}}\left(\begin{array}{c}
\mathbf{a}\\
\mathbf{b}\end{array}\right)$ the unitary vector of the line $\Delta$, so we have \[
\left\Vert \pi_{\Delta}\left(ne_{1}+me_{2}\right)\right\Vert _{\mathbb{R}^{2}}=\left|\left\langle u,ne_{1}+me_{2}\right\rangle _{\mathbb{R}^{2}}\right|\]
\[
=\left|nu_{1}+mu_{2}\right|=\left|u_{1}\right|\left|n+m\theta\right|;\]
since $\theta$ is a Roth number, for all $\varepsilon>0$ there exist
$C_{\varepsilon}>0$ such that \[
\left\Vert \pi_{\Delta}\left(ne_{1}+me_{2}\right)\right\Vert _{\mathbb{R}^{2}}\geq\left|u_{1}\right|\frac{C_{\varepsilon}}{|m|^{1+\varepsilon}}\]
\[
\geq\frac{\left|u_{1}\right|C_{\varepsilon}}{\left\Vert ne_{1}+me_{2}\right\Vert _{\mathbb{R}^{2}}^{1+\varepsilon}}.\]
In a similar way we ge\[
\left\Vert \pi_{\Gamma}\left(ne_{1}+me_{2}\right)\right\Vert _{\mathbb{R}^{2}}\geq\left|u_{2}\right|\frac{C_{\varepsilon}}{|n|^{1+\varepsilon}}\geq\frac{\left|u_{2}\right|C_{\varepsilon}}{\left\Vert ne_{1}+me_{2}\right\Vert _{\mathbb{R}^{2}}^{1+\varepsilon}};\]
therefore with $K_{\varepsilon}:=\min\left(\left|u_{1}\right|C_{\varepsilon},\left|u_{2}\right|C_{\varepsilon}\right)\leq C_{\varepsilon}$
we obtain the lemma.
\end{proof}
A consequence of this lemma is :
\begin{thm}
For all $\varepsilon>0$ there $K_{\varepsilon}\leq C_{\varepsilon}$,
here $C_{\varepsilon}$ denotes the Roth constant of $\theta=\frac{b}{a}$,
such that for all $t\geq0$\[
d\left(\varphi_{t},\mathbb{Z}_{*}^{2}\right)\geq\frac{K_{\varepsilon}}{\left(\frac{\sqrt{2}}{2}+t\sqrt{\mathbf{a}^{2}+\mathbf{b}^{2}}\right)^{1+\varepsilon}}.\]
\end{thm}
\begin{proof}
We observe that for all $t\geq0$ the point $\varphi_{t}$ belongs
to the line $\Delta$, thus there exists a pair $\left(n_{t},m_{t}\right)\in\mathbb{Z}_{*}^{2}$
such that\[
d\left(\varphi_{t},\mathbb{Z}_{*}^{2}\right)=\left\Vert \overrightarrow{O\varphi_{t}}-\left(n_{t}e_{1}+m_{t}e_{2}\right)\right\Vert _{\mathbb{R}^{2}}\]
\[
\geq\left\Vert \pi_{\Gamma}\left(n_{t}e_{1}+m_{t}e_{2}\right)\right\Vert _{\mathbb{R}^{2}};\]
applying the lemma above we get\[
d\left(\varphi_{t},\mathbb{Z}_{*}^{2}\right)\geq\frac{K_{\varepsilon}}{\left\Vert n_{t}e_{1}+m_{t}e_{2}\right\Vert _{\mathbb{R}^{2}}^{1+\varepsilon}}.\]
On the other hand we have the majorization\[
\left\Vert \overrightarrow{O\varphi_{t}}-\left(n_{t}e_{1}+m_{t}e_{2}\right)\right\Vert _{\mathbb{R}^{2}}\leq\frac{\sqrt{2}}{2},\]
and, by triangular inequality we obtain\foreignlanguage{english}{\[
\left\Vert n_{t}e_{1}+m_{t}e_{2}\right\Vert _{\mathbb{R}^{2}}\leq\left\Vert \overrightarrow{O\varphi_{t}}\right\Vert _{\mathbb{R}^{2}}+\frac{\sqrt{2}}{2}.\]
}Therefore, since $\left\Vert \overrightarrow{O\varphi_{t}}\right\Vert _{\mathbb{R}^{2}}=t\sqrt{\mathbf{a}^{2}+\mathbf{b}^{2}},$
we get for all $t\geq0,\,\varepsilon>0$ \[
d\left(\varphi_{t},\mathbb{Z}_{*}^{2}\right)\geq\frac{K_{\varepsilon}}{\left(t\sqrt{\mathbf{a}^{2}+\mathbf{b}^{2}}+\frac{\sqrt{2}}{2}\right)^{1+\varepsilon}}.\]
\end{proof}
\begin{cor}
For all $\varepsilon>0$ and for every $\eta\in\left]0,\frac{\sqrt{2}^{1+\varepsilon}}{2}\right[\subset\left]0,\frac{\sqrt{2}}{2}\right[$
we have :\[
d\left(\varphi_{t},\mathbb{Z}_{*}^{2}\right)<\eta\Rightarrow t>\frac{1}{\sqrt{\mathbf{a}^{2}+\mathbf{b}^{2}}}\left(\left(\frac{K_{\varepsilon}}{\eta}\right)^{\frac{1}{1+\varepsilon}}-\frac{\sqrt{2}}{2}\right).\]
\end{cor}
\begin{proof}
Suppose $d\left(\varphi_{t}(0),\mathbb{Z}_{*}^{2}\right)<\eta$, it
then follows from the theorem above that for all $\varepsilon>0$
there exists a constant $K_{\varepsilon}\in\left]0,\frac{1}{2}\right[$
such that \[
\frac{K_{\varepsilon}}{\left(\frac{\sqrt{2}}{2}+t\sqrt{\mathbf{a}^{2}+\mathbf{b}^{2}}\right)^{1+\varepsilon}}<\eta\]
holds for all $t\geq0$; i.e.\[
\left(\frac{K_{\varepsilon}}{\eta}\right)^{\frac{1}{1+\varepsilon}}<\frac{\sqrt{2}}{2}+t\sqrt{\mathbf{a}^{2}+\mathbf{b}^{2}}.\]
\end{proof}
\begin{rem}
Since $\eta\in\left]0,\frac{\sqrt{2}^{1+\varepsilon}}{2}\right[\subset\left]0,\frac{\sqrt{2}}{2}\right[$
we have $\left(\frac{K_{\varepsilon}}{\eta}\right)^{\frac{1}{1+\varepsilon}}\geq\frac{\sqrt{2}}{2}.$\end{rem}
\begin{notation}
For $\varepsilon>0$ and $\eta>0,$ let us denote :\[
t_{\eta}(\varepsilon):=\frac{1}{\sqrt{\mathbf{a}^{2}+\mathbf{b}^{2}}}\left(\left(\frac{K_{\varepsilon}}{\eta}\right)^{\frac{1}{1+\varepsilon}}\mathbf{-}\frac{\sqrt{2}}{2}\right).\]
\end{notation}
\begin{thm}
For all $\varepsilon>0$, for every $\eta\in\left]0,\frac{\sqrt{2}^{1+\varepsilon}}{2}\right[$
with $\eta$ small enough such that $t_{\eta}(\varepsilon)\geq\max T_{cl_{i}}$
and for all $k\geq1;$ there exists a constant $D_{k}>0$ which does
not depend on $h$ such that the inequality \[
\left|\mathbf{a}_{\mathbf{1}}(t)\right|\leq D_{k}h^{k(1-\max\delta_{i}^{\prime})}\eta^{-k}\]
holds for all $t\in$$[\max T_{cl_{i}},t_{\eta}(\varepsilon)].$ \end{thm}
\begin{proof}
Our starting point is that for all $t\geq\max T_{cl_{i}}$ we have
\[
d\left(t,T_{cl_{i}}\mathbb{Z}\right)=d\left(t,T_{cl_{i}}\mathbb{N}^{*}\right).\]
Next, since $T_{cl_{1}}=\left|\frac{2\pi}{a}\right|=\left|\frac{1}{\mathbf{a}}\right|$
and $T_{cl_{2}}=\left|\frac{2\pi}{b}\right|=\left|\frac{1}{\mathbf{b}}\right|$
we get \[
d\left(t,T_{cl_{1}}\mathbb{Z}\right)=d\left(\mathbf{a}t,\mathbb{N}^{*}\right),\; d\left(t,T_{cl_{2}}\mathbb{Z}\right)=d\left(\mathbf{b}t,\mathbb{N}^{*}\right).\]
Therefore, from the corrolary above (by contraposed) we obtain for
all $t\in$$[\max T_{cl_{i}},t_{\eta}(\varepsilon)]$\foreignlanguage{english}{\[
d\left((\mathbf{a}t,\mathbf{b}t),\mathbb{Z}_{*}^{2}\right)\geq\eta;\]
}and since the norms $\left\Vert (x,y)\right\Vert _{\mathbb{R}^{2}}$
and $\left|x\right|+\left|y\right|$ are equivalent on $\mathbb{R}^{2}$,
there exists a constant $C>0$ such that \[
d\left(t,T_{cl_{1}}\mathbb{Z}\right)+d\left(t,T_{cl_{2}}\mathbb{Z}\right)\geq C\eta\]
holds for all $t\in$$[\max T_{cl_{i}},t_{\eta}(\varepsilon)]$. 

Next, since the function $\mathfrak{F}\left(\chi^{2}\right)$ belongs
to the space $\mathcal{S}(\mathbb{R}^{2})$ we have \[
\forall k\in\mathbb{N}^{2},\,\exists M_{k}>0,\,\forall\zeta_{1},\zeta_{2}\in\mathbb{R}^{2},\,\left|\mathfrak{F}\left(\chi^{2}\right)\left(\zeta_{1},\zeta_{2}\right)\right|\leq\frac{M_{k}}{\left(\left|\zeta_{1}\right|+\left|\zeta_{2}\right|\right)^{k}};\]
thus we obtain for all $t\in$$[\max T_{cl_{i}},t_{\eta}(\varepsilon)]$\textit{\[
\left|\mathfrak{F}\left(\chi^{2}\right)\left(-\frac{h^{\delta_{1}^{\prime}-1}}{\omega_{1}}d\left(t,T_{cl_{1}}\mathbb{Z}\right),-\frac{h^{\delta_{2}^{\prime}-1}}{\omega_{2}}d\left(t,T_{cl_{2}}\mathbb{Z}\right)\right)\right|\]
}\[
\leq\frac{M_{k}}{\left(\frac{h^{\delta_{1}^{\prime}-1}}{\omega_{1}}d\left(t,T_{cl_{1}}\mathbb{Z}\right)+\frac{h^{\delta_{2}^{\prime}-1}}{\omega_{2}}d\left(t,T_{cl_{2}}\mathbb{Z}\right)\right)^{k}}\]
\[
\leq\frac{{\displaystyle \max(\omega_{i})^{k}}M_{k}}{h^{k\max(\delta_{i}^{\prime})-k}\left(d\left(t,T_{cl_{1}}\mathbb{Z}\right)+d\left(t,T_{cl_{2}}\mathbb{Z}\right)\right)^{k}}\]
\[
\leq\frac{{\displaystyle \max(\omega_{i})^{k}}M_{k}}{h^{k\max(\delta_{i}^{\prime})-k}}\frac{1}{C^{k}\eta^{k}}={\displaystyle \max(\omega_{i})^{k}\frac{M^{k}}{C^{k}}}h^{k(1-\max(\delta_{i}^{\prime}))}\eta^{-k}.\]

\end{proof}
Applying this theorem with $\eta=h^{s}$ where the real number $s$
belongs to $\left]0,1-\max\delta_{i}^{\prime}\right[$, we have :
\begin{cor}
For all $\varepsilon>0$, $s\in\left]0,1-\max\delta_{i}^{\prime}\right[$
and for $h$ small enough such that $t_{h^{s}}(\varepsilon)\geq\max T_{cl_{i}};$
the following equality\[
\mathbf{a}_{1}(t)=O(h^{\infty})\]
holds uniformly for all $t\in[\max T_{cl_{i}},t_{h^{s}}(\varepsilon)].$ 
\end{cor}

\subsection*{Notes on time scales }

From a pratical point of view, we must verify that for all $\varepsilon\geq0$
\[
t_{h^{s}}(\varepsilon)\leq h^{\alpha}\]
where \foreignlanguage{english}{$\alpha>1-2\min\left(\delta_{i}\right).$
Indeed we have :}
\begin{prop}
Suppose $\min\delta_{i}>\frac{2}{3}$, for all $\varepsilon>0$ and
for all $s\in\left]0,1-\max\delta_{i}^{\prime}\right[$ we have \[
t_{h^{s}}(\varepsilon)\leq\frac{1}{2\sqrt{\mathbf{a}^{2}+\mathbf{b}^{2}}}\left(h^{1-2{\displaystyle \min\delta_{i}}}-\sqrt{2}\right).\]
\end{prop}
\begin{proof}
For all $s>0$ and for all $\varepsilon>0$ , since \[
t_{h^{S}}(\varepsilon)=\frac{1}{\sqrt{\mathbf{a}^{2}+\mathbf{b}^{2}}}\left(\left(K_{\varepsilon}\right)^{\frac{1}{1+\varepsilon}}h^{-\frac{s}{1+\varepsilon}}-\frac{\sqrt{2}}{2}\right)\]
for $h\rightarrow0$ we have the equivalence \[
t_{h^{s}}\sim D_{\varepsilon}h^{-\frac{s}{1+\varepsilon}}\]
where $D_{\varepsilon}:=\frac{1}{\sqrt{\mathbf{a}^{2}+\mathbf{b}^{2}}}\left(K_{\varepsilon}\right)^{\frac{1}{1+\varepsilon}}>0$.
On the other hand, we see that for all $\varepsilon>0$ we have $1-2{\displaystyle \min\delta_{i}\leq}\frac{\max\delta_{i}^{\prime}-1}{1+\varepsilon}.$
Hence \[
h^{\frac{\max\delta_{i}^{\prime}-1}{1+\varepsilon}}\leq h^{1-2{\displaystyle \min\delta_{i}}}.\]
Therefore, we obtain \[
t_{h^{S}}(\varepsilon)\leq\frac{1}{\sqrt{\mathbf{a}^{2}+\mathbf{b}^{2}}}\left(\left(K_{\varepsilon}\right)^{\frac{1}{1+\varepsilon}}h^{\frac{\max\delta_{i}^{\prime}-1}{1+\varepsilon}}-\frac{\sqrt{2}}{2}\right)\]
\[
\leq\frac{1}{\sqrt{\mathbf{a}^{2}+\mathbf{b}^{2}}}\left(\left(K_{\varepsilon}\right)^{\frac{1}{1+\varepsilon}}h^{1-2{\displaystyle \min\delta_{i}}}-\frac{\sqrt{2}}{2}\right)\leq\frac{1}{\sqrt{\mathbf{a}^{2}+\mathbf{b}^{2}}}\left(\frac{1}{2}h^{1-2{\displaystyle \min\delta_{i}}}-\frac{\sqrt{2}}{2}\right).\]

\end{proof}

\subsection*{Use of continued fractions}

Now, we can wonder what are the accurate times when $d\left(\varphi_{t},\mathbb{Z}_{*}^{2}\right)<\eta$
? To solve this problem we will use the continued fraction theory.

\subsubsection*{Some useful theorems}

The continued fractions are essentially used for the approximation
of real numbers. There exists two types of continued fractions: the
finite continued fractions representing rational numbers and the infinite
continued fractions representing irrational numbers. For all irrational
number $\theta,$ there exists a pair sequence $\left(q_{n},p_{n}\right)\in$$\mathbb{N}^{2}$
such that \[
\left|\theta-\frac{p_{n}}{q_{n}}\right|\leq\frac{1}{q_{n}^{2}}\]
holds for all $n\geq0$. This sequence is given by the continued fractions
algorithm (see \textbf{{[}Ro-Sz{]}, {[}Khi{]}}). Geometrically speaking,
the construction principle for this sequence is as follows (see \textbf{{[}Arn2{]}})
: consider $v_{0}:=(0,1)$ and $v_{-1}:=(1,0)$. It is obvious that
these points lie on different sides of the line $y=\theta x$. By
induction : let the vectors $v_{k-1}$ and $v_{k}$ be constructed
whereas to construct the new vector $v_{k+1}$, we add to the vector
$v_{k-1}$ the vector $v_{k}$ as many times as we can in such a way
the new vector $v_{k+1}$ lies on the same side of the line $y=\theta x$
as the vector $v_{k-1}$ :

\[
v_{k+1}=a_{k}v_{k}+v_{k-1}\]
i.e\[
\left\{ \begin{array}{cc}
q_{k+1}=a_{k}q_{k}+q_{k-1}\\
\\p_{k+1}=a_{k}p_{k}+p_{k-1}\end{array}\right.\]
where $\left(a_{k}\right)_{k\geq0}$ is a sequence of integers strictly
$>0$. 

\vspace {0.25cm}

\vspace {0.25cm}

We note that the sequence $\left(q_{n}\right)_{n}$ is strictly increasing.
With the standard notation continued fractions we have \foreignlanguage{english}{\[
\left[a_{0},a_{1},\ldots,a_{n},\ldots\right]:=a_{0}+\frac{1}{a_{1}+\frac{1}{a_{2}+\frac{1}{a_{3}+\ldots}}}\]
}end we have (see for example \textbf{{[}Khi{]}}) the relation \foreignlanguage{english}{$\left[a_{0},a_{1},\ldots,a_{n}\right]=\frac{p_{n}}{q_{n}}.$}
\begin{example}
The number $\pi$ is given by : $\pi=\left[3,7,15,1,292,1\ldots\right].$
\end{example}

\subsubsection*{Approach time }

Let us denote by $D$ the line $y=\frac{\mathbf{b}}{\mathbf{a}}x=\theta x$.
Hence, for all $t\geq0$ we have : \[
\left[O\varphi_{t}\right]\subset D.\]
For a fixed $n\geq0$ we wish to find the point $M_{n}$ of $D$ such
that $d\left(M_{n},\left(q_{n},p_{n}\right)\right)=d\left(D,\left(q_{n},p_{n}\right)\right).$
In other words, we wish to find the time $\tau_{n}$ such that $d\left(\varphi_{\tau_{n}},\left(q_{n},p_{n}\right)\right)=d\left(D,\left(q_{n},p_{n}\right)\right).$
\begin{prop}
For all $n\geq1$ the unique $\tau_{n}\geq0$ such that $d\left(\varphi_{\tau_{n}},\left(q_{n},p_{n}\right)\right)=d\left(D,\left(q_{n},p_{n}\right)\right)$
is given by \[
\tau_{n}=\frac{\mathbf{a}q_{n}+\mathbf{b}p_{n}}{\mathbf{a}^{2}+\mathbf{b}^{2}}.\]
Moreover we have \[
d\left(\varphi_{\tau_{n}},\mathbb{Z}_{*}^{2}\right)\leq\frac{\mathbf{a}}{\sqrt{\mathbf{a}^{2}+\mathbf{b}^{2}}}\frac{1}{q_{n}}<\frac{1}{q_{n}}.\]
\end{prop}
\begin{proof}
For $n\neq0$ fixed, we want to find $t\geq0$ such that $d\left(\varphi_{t},\left(q_{n},p_{n}\right)\right)=d\left(D,\left(q_{n},p_{n}\right)\right).$
This means that we want to find $t\geq0$ such that \[
\overrightarrow{O\varphi_{t}}\perp\left(\left(q_{n}e_{1}+p_{n}e_{2}\right)\mathbf{-}\overrightarrow{O\varphi_{t}}\right)\]
i.e. : to find $t\geq0$ such that\[
\left\langle \overrightarrow{O\varphi_{t}},\left(q_{n}e_{1}+p_{n}e_{2}\right)\mathbf{-}\overrightarrow{O\varphi_{t}}\right\rangle _{\mathbb{R}^{2}}=0.\]
Consequently we solve the equation $\mathbf{a}t(q_{n}-\mathbf{a}t)+\mathbf{b}t(p_{n}-\mathbf{b}t)=0$
and we find for non-null solution : $t=\frac{\mathbf{a}q_{n}+\mathbf{b}p_{n}}{\mathbf{a}^{2}+\mathbf{b}^{2}}.$
Therefore, at time $t=\tau_{n}:=\frac{\mathbf{a}q_{n}+\mathbf{b}p_{n}}{\mathbf{a}^{2}+\mathbf{b}^{2}}$
we obtain \[
d\left(\varphi_{\tau_{n}},\left(q_{n},p_{n}\right)\right)=d\left(\left(\frac{\mathbf{a}^{2}q_{n}+\mathbf{ab}p_{n}}{\mathbf{a}^{2}+\mathbf{b}^{2}},\frac{\mathbf{ab}q_{n}+\mathbf{b}^{2}p_{n}}{\mathbf{a}^{2}+\mathbf{b}^{2}}\right),\left(q_{n},p_{n}\right)\right)\]
\[
=\frac{1}{\mathbf{a}^{2}+\mathbf{\mathbf{b}^{2}}}\sqrt{\mathbf{b}^{2}\left(\mathbf{a}p_{n}-\mathbf{b}q_{n}\right)^{2}+\mathbf{a}^{2}\left(\mathbf{b}q_{n}-\mathbf{a}p_{n}\right)^{2}}.\]
Since for all integer $n$ we know that $\left|\theta-\frac{p_{n}}{q_{n}}\right|\leq\frac{1}{q_{n}^{2}}$,
i.e. $\left|q_{n}\mathbf{b}-p_{n}\mathbf{a}\right|\leq\frac{\mathbf{a}}{q_{n}}$
holds for all integer $n$, so we deduce that \[
d\left(\varphi_{\tau_{n}},\left(q_{n},p_{n}\right)\right)\leq\frac{1}{\mathbf{a}^{2}+\mathbf{b}^{2}}\sqrt{\mathbf{b}^{2}\frac{\mathbf{a}^{2}}{q_{n}^{2}}+\mathbf{a}^{2}\frac{\mathbf{a}^{2}}{q_{n}^{2}}}\leq\frac{\mathbf{a}}{\sqrt{\mathbf{a}^{2}+\mathbf{b}^{2}}}\frac{1}{q_{n}}<\frac{1}{q_{n}}.\]
For conclude, we note that \[
d\left(\varphi_{\tau_{n}},\left(q_{n},p_{n}\right)\right)\geq d\left(\varphi_{\tau_{n}},\mathbb{Z}_{*}^{2}\right)\]
holds for all integer $n.$
\end{proof}
We wish to generalize this result : we wish to analyze the behaviour
of the distance between the set $\mathbb{Z}_{*}^{2}$ and the flow
$\varphi_{t}$ when $t$ is in a neighbourhood of the time $\tau_{n}$.
\begin{notation}
For $r>0$ let us denotes by $B(\tau_{n},r)$ the closed ball of center
$\tau_{n}$ and radius $r>0$ :\[
B(\tau_{n},r):=\left[\frac{\mathbf{a}q_{n}+\mathbf{b}p_{n}}{\mathbf{a}^{2}+\mathbf{b}^{2}}-r,\frac{\mathbf{a}q_{n}+\mathbf{b}p_{n}}{\mathbf{a}^{2}+\mathbf{b}^{2}}+r\right].\]
\end{notation}
\begin{prop}
For all $r>0$ \[
d\left(\varphi_{t},\mathbb{Z}_{*}^{2}\right)\leq\frac{1}{\mathbf{a}^{2}+\mathbf{b}^{2}}\sqrt{\left(\left(\frac{\mathbf{ab}}{q_{n}}\right)+r\mathbf{a}(\mathbf{a}^{2}+\mathbf{b}^{2})\right)^{2}+\left(\frac{\mathbf{a^{2}}}{q_{n}}+r\mathbf{b}(\mathbf{a}^{2}+\mathbf{b}^{2})\right)^{2}}\]
holds for all $t\in B(\tau_{n},r).$\end{prop}
\begin{proof}
We begin with the following inequality : for all $t\geq0$ and for
all $n$ \[
d\left(\varphi_{t},\mathbb{Z}_{*}^{2}\right)\leq d\left(\varphi_{t},\left(q_{n},p_{n}\right)\right).\]
Next, it's clear that for all $t\in B(\tau_{n},r)$ we have\[
\varphi_{t}\in B\left(\frac{\mathbf{a}^{2}q_{n}+\mathbf{ab}p_{n}}{\mathbf{a}^{2}+\mathbf{b}^{2}},r\mathbf{a}\right)\times B\left(\frac{\mathbf{ab}q_{n}+\mathbf{b}^{2}p_{n}}{\mathbf{a}^{2}+\mathbf{b}^{2}},r\mathbf{b}\right).\]
Therefore, for all $t\in B(\tau_{n},r)$ we have \[
d\left(\varphi_{t},\left(q_{n},p_{n}\right)\right)\]
\[
\leq\sqrt{\left(\frac{\mathbf{a}^{2}q_{n}+\mathbf{ab}p_{n}}{\mathbf{a}^{2}+\mathbf{b}^{2}}+r\mathbf{a}-q_{n}\right)^{2}+\left(\frac{\mathbf{ab}q_{n}+\mathbf{b}^{2}p_{n}}{\mathbf{a}^{2}+\mathbf{b}^{2}}+r\mathbf{b}-p_{n}\right)^{2}}\]
\[
=\frac{1}{\mathbf{a}^{2}+\mathbf{b}^{2}}\sqrt{\left(\mathbf{ab}p_{n}-\mathbf{b}^{2}q_{n}+r\mathbf{a}(\mathbf{a}^{2}+\mathbf{b}^{2})\right)^{2}+\left(\mathbf{ab}q_{n}-\mathbf{a}^{2}p_{n}+r\mathbf{b}(\mathbf{a}^{2}+\mathbf{b}^{2})\right)^{2}}\]
\[
\leq\frac{1}{\mathbf{a}^{2}+\mathbf{b}^{2}}\sqrt{\left(\left(\frac{\mathbf{ab}}{q_{n}}\right)+r\mathbf{a}(\mathbf{a}^{2}+\mathbf{b}^{2})\right)^{2}+\left(\frac{\mathbf{a^{2}}}{q_{n}}+r\mathbf{b}(\mathbf{a}^{2}+\mathbf{b}^{2})\right)^{2}}.\]
\end{proof}
\begin{rem}
For $r=0$ we obtain\[
d\left(\varphi_{t},\mathbb{Z}_{*}^{2}\right)\leq\frac{1}{\mathbf{a}^{2}+\mathbf{b}^{2}}\sqrt{\left(\left(\frac{\mathbf{ab}}{q_{n}}\right)\right)^{2}+\left(\frac{\mathbf{a^{2}}}{q_{n}}\right)^{2}}\leq\frac{\mathbf{a}}{\sqrt{\mathbf{a}^{2}+\mathbf{b}^{2}}}\frac{1}{q_{n}}.\]
and we obtain again the result of the proposition 3.28.
\end{rem}
Now, let us give an asymptotic equivalent (for $n\rightarrow\infty$)
of the real number $\tau_{n}$ :
\begin{prop}
For $n\rightarrow\infty$ we have $\tau_{n}\sim\Omega q_{n}$; where
$\Omega:=\frac{\mathbf{a}+\mathbf{b}\theta}{\mathbf{a}^{2}+\mathbf{b}^{2}}>0.$ \end{prop}
\begin{proof}
We just write the fraction $\tau_{n}/\Omega q_{n}$ :\[
\frac{\tau_{n}}{\Omega q_{n}}=\frac{\mathbf{a}q_{n}+\mathbf{b}p_{n}}{\mathbf{a}^{2}+\mathbf{b}^{2}}\frac{\mathbf{a}^{2}+\mathbf{b}^{2}}{q_{n}(\mathbf{a}+\mathbf{b}\theta)}=\frac{\mathbf{a}}{\mathbf{a}+\mathbf{b}\theta}+\frac{\mathbf{b}p_{n}}{q_{n}(\mathbf{a}+\mathbf{b}\theta)}\]
and since $\lim_{n\rightarrow\infty}p_{n}/q_{n}=\theta$ we obtain
that  $\lim_{n\rightarrow\infty}\tau_{n}/\Omega q_{n}=1.$
\end{proof}
Now, let us come back to the autocorrelation function approximation
$\mathbf{a_{1}}$. Start by a notation and a remark : 
\begin{notation}
For $\mu>0$ let us denotes by $\mathcal{A}_{h}=\mathcal{A}_{h}(\theta,\mu)$
the following set :\[
\mathcal{A}_{h}:=\left\{ q_{n}\in\mathbb{N},\, q_{n}\in\left[h^{\min\delta_{i}^{\prime}-1+\mu},h^{1-2\min\delta_{i}-\mu}\right]\right\} .\]
\end{notation}
\begin{rem}
For $\mu>0$, we have of course $\left[h^{\min\delta_{i}^{\prime}-1-\mu},h^{1-2\min\delta_{i}+\mu}\right]\subset\left[h^{\min\delta_{i}^{\prime}-1},h^{1-2\min\delta_{i}}\right].$

If we suppose that the set $\mathcal{A}_{h}$ is non empty, we have
some periods for the function $\mathbf{a_{1}}$, indeed we have : \end{rem}
\begin{thm}
Suppose $\mathcal{A}_{h}\neq\emptyset,$ then\[
\sup_{n\in\{m\in\mathbb{N},\, q_{m}(\theta)\in\mathcal{A}_{h}\}}\left|\mathbf{\mathbf{a_{1}}}(\tau_{n})-1\right|=O(h^{\mu}).\]
\end{thm}
\begin{proof}
Applying the Taylor-Lagrange formula on the function $(x,y)\mapsto\mathfrak{F}\left(\chi^{2}\right)(x,y)$
near the origin : for all $t\geq0$ there exists $\theta=\theta\left(t,h,T_{cl_{1}},T_{cl_{2}}\right)\in\left]0,1\right[$
such that\[
\mathfrak{F}\left(\chi^{2}\right)\left(-\frac{h^{\delta_{1}^{\prime}-1}}{\omega_{1}}d\left(T_{cl_{1}}\mathbb{Z},t\right),-\frac{h^{\delta_{2}^{\prime}-1}}{\omega_{2}}d\left(T_{cl_{2}}\mathbb{Z},t\right)\right)=\mathfrak{F}\left(\chi^{2}\right)\left(0,0\right)\]
\[
-\frac{h^{\delta_{1}^{\prime}-1}}{\omega_{1}}d\left(T_{cl_{1}}\mathbb{Z},t\right)\frac{\partial\mathfrak{F}\left(\chi^{2}\right)}{\partial x}\left(\theta\frac{h^{\delta_{1}^{\prime}-1}}{\omega_{1}}d\left(T_{cl_{1}}\mathbb{Z},t\right),\theta\frac{h^{\delta_{2}^{\prime}-1}}{\omega_{2}}d\left(T_{cl_{2}}\mathbb{Z},t\right)\right)\]

\[
-\frac{h^{\delta_{2}^{\prime}-1}}{\omega_{2}}d\left(T_{cl_{2}}\mathbb{Z},t\right)\frac{\partial\mathfrak{F}\left(\chi^{2}\right)}{\partial y}\left(\theta\frac{h^{\delta_{1}^{\prime}-1}}{\omega_{1}}d\left(T_{cl_{1}}\mathbb{Z},t\right),\theta\frac{h^{\delta_{2}^{\prime}-1}}{\omega_{2}}d\left(T_{cl_{2}}\mathbb{Z},t\right)\right).\]
We know that for all $n\geq1$ the distance between the part of the
flow $\varphi_{\tau_{n}}$ and the set $\mathbb{Z}_{*}^{2}$ is strictly
lower than $\frac{1}{q_{n}}$. Hence, if we suppose that from a certain
point, like $n\geq N$ , the sequence $\left(\frac{1}{q_{n}}\right)_{n}$
is strictly lower than $h^{s}$ (with $s>0)$, then we obtain the
majorization $d\left(\varphi_{\tau_{n}},\mathbb{Z}_{*}^{2}\right)\leq h^{s}$.
Therefore, for all $i\in\{1,2\}$ we have also $d\left(\tau_{n},T_{cl_{i}}\mathbb{Z}\right)\leq h^{s}.$
Consequently for all $i\in\{1,2\}$ we get \foreignlanguage{english}{$h^{\delta_{i}^{\prime}-1}d\left(\tau_{n},T_{cl_{i}}\mathbb{Z}\right)\leq h^{\delta_{i}^{\prime}-1+s}$
}. Since we suppose $s\geq-\min\delta_{i}^{\prime}+1+\mu$ with $\mu>0$
we deduce that : \foreignlanguage{english}{\[
h^{\delta_{i}^{\prime}-1}d\left(\tau_{n},T_{cl_{i}}\mathbb{Z}\right)\leq h^{\mu};\]
}and, for $h$ small enough, we obtain \[
\left|\frac{h^{\delta_{1}^{\prime}-1}}{\omega_{1}}d\left(T_{cl_{1}}\mathbb{Z},\tau_{n}\right)\frac{\partial\mathfrak{F}\left(\chi^{2}\right)}{\partial x}\left(\theta\frac{h^{\delta_{1}^{\prime}-1}}{\omega_{1}}d\left(T_{cl_{1}}\mathbb{Z},\tau_{n}\right),\theta\frac{h^{\delta_{2}^{\prime}-1}}{\omega_{2}}d\left(T_{cl_{2}}\mathbb{Z},t\tau_{n}\right)\right)\right|\leq Mh^{\mu},\]
\[
\left|\frac{h^{\delta_{2}^{\prime}-1}}{\omega_{2}}d\left(T_{cl_{2}}\mathbb{Z},\tau_{n}\right)\frac{\partial\mathfrak{F}\left(\chi^{2}\right)}{\partial y}\left(\theta\frac{h^{\delta_{1}^{\prime}-1}}{\omega_{1}}d\left(T_{cl_{1}}\mathbb{Z},\tau_{n}\right),\theta\frac{h^{\delta_{2}^{\prime}-1}}{\omega_{2}}d\left(T_{cl_{2}}\mathbb{Z},t\tau_{n}\right)\right)\right|\leq Nh^{\mu};\]
where $M,\, M^{\prime}>0$ are constant which does not depend on $h$.
Next, it comes from the Taylor formula written above that \[
\sup_{n\in\{m\in\mathbb{N},\, q_{m}(\theta)\in\mathcal{A}_{h}\}}\left|\mathbf{\mathbf{a_{1}}}(\tau_{n})-1\right|=\frac{1}{\mathfrak{F}\left(\chi^{2}\right)(0,0)}\]
\[
\sup_{n\in\{m\in\mathbb{N},\, q_{m}(\theta)\in\mathcal{A}_{h}\}}\left|\frac{h^{\delta_{1}^{\prime}-1}}{\omega_{1}}d\left(T_{cl_{1}}\mathbb{Z},\tau_{n}\right)\frac{\partial\mathfrak{F}\left(\chi^{2}\right)}{\partial x}\left(\theta\frac{h^{\delta_{1}^{\prime}-1}}{\omega_{1}}d\left(T_{cl_{1}}\mathbb{Z},\tau_{n}\right),\theta\frac{h^{\delta_{2}^{\prime}-1}}{\omega_{2}}d\left(T_{cl_{2}}\mathbb{Z},\tau_{n}\right)\right)\right.\]
\[
\left.\frac{h^{\delta_{2}^{\prime}-1}}{\omega_{2}}d\left(T_{cl_{2}}\mathbb{Z},\tau_{n}\right)\frac{\partial\mathfrak{F}\left(\chi^{2}\right)}{\partial y}\left(\theta\frac{h^{\delta_{1}^{\prime}-1}}{\omega_{1}}d\left(T_{cl_{1}}\mathbb{Z},\tau_{n}\right),\theta\frac{h^{\delta_{2}^{\prime}-1}}{\omega_{2}}d\left(T_{cl_{2}}\mathbb{Z},\tau_{n}\right)\right)\right|\]
\[
\leq h^{\mu}M.\]

\end{proof}

\subsection*{Counting of the sequence $q_{n}$}

In view of the theorem above, let us now tackle an important problem
: what is the cardinality of the set $\mathcal{A}_{h}$ ? Note that
since the sequence $\left(q_{n}(\theta)\right)_{n\in\mathbb{N}}$
is strictly increasing we have \[
\#\left\{ \mathcal{A}_{h}\right\} =\#\left\{ n\in\mathbb{N},\, q_{n}(\theta)\in\mathcal{A}_{h}\right\} .\]
Start by a simple majorization of the integer $\#\left\{ \mathcal{A}_{h}\right\} $
:\[
\#\left\{ \mathcal{A}_{h}\right\} \leq\#\left\{ \mathbb{N}\cap\left[h^{\min\delta_{i}^{\prime}-1-\mu},h^{1-2\min\delta_{i}+\mu}\right]\right\} \leq E\left[\delta(h)\right]+1\]
where $\delta(h):=h^{1-2\min\delta_{i}+\mu}-h^{\min\delta_{i}^{\prime}-1-\mu}$
and $E[x]$ denotes the integer part of $x$. Then, for $h\rightarrow0$
we have the equivalence $\delta(h)\sim h^{1-2\min\delta_{i}+\mu}.$
So we get a majorization of the integer $\#\left\{ \mathcal{A}_{h}\right\} $
in order $h^{1-2\min\delta_{i}+\mu}$.

Nevertheless, find a minoration of the integer $\#\left\{ \mathcal{A}_{h}\right\} $
is more difficult; but it's cleat that for all $n^{*}>1$ there exists
$h^{*}\in\left]0,1\right[$ such that :\[
\left[h^{^{*}\min\delta_{i}^{\prime}-1-\mu},h^{*1-2\min\delta_{i}+\mu}\right]{\displaystyle \bigcap}\left({\displaystyle \bigcup_{n=0}^{+\infty}}q_{n}(\theta)\right)=\left\{ q_{n^{*}}(\theta)\right\} .\]
In order to estimate the integer $\#\left\{ \mathcal{A}_{h}\right\} ,$
we must know the distribution of the sequence $\left(q_{n}(\theta)\right)_{n\in\mathbb{N}}$
on the real axis (in particular on the compact set $\left[h^{\min\delta_{i}^{\prime}-1-\mu},h^{1-2\min\delta_{i}+\mu}\right]$)
depending on the number $\theta.$ Let's try to give some distribution
examples of the sequance $\left(q_{n}(\theta)\right)_{n\in\mathbb{N}}$
.

\subsubsection*{An exemple : the golden ratio}

The golden ratio $\varphi$ is the unique real roots of $X^{2}-X-1=0$,
i.e. $\varphi=(1+\sqrt{5})/2$. The continued fraction of the golden
ration is :\foreignlanguage{english}{\[
\varphi=\left[1,1,\ldots,1,\ldots\right]=1+\frac{1}{1+\frac{1}{1+\frac{1}{1+\ldots}}}.\]
}Consequently the golden ratio is that one of the most difficult real
number to approximate with rationals numbers. An another particularity
of the golden ratio is that the sequence of the denominators $\left(q_{n}\right)_{n}$
from the continued fraction algorithm is equal to the Fibonacci sequence
$\left(\mathbb{F}_{n}\right)_{n}$ :\[
\mathbb{F}_{n}:=\frac{1}{\sqrt{5}}\left(\frac{1+\sqrt{5}}{2}\right)^{n}-\frac{1}{\sqrt{5}}\left(\frac{1-\sqrt{5}}{2}\right)^{n}.\]
We note that \[
\lim_{n\rightarrow+\infty}\frac{\mathbb{F}_{n+1}}{\mathbb{F}_{n}}=\varphi.\]
Next for $n\rightarrow+\infty$ we have also :\[
\mathbb{F}_{n}\sim\frac{1}{\sqrt{5}}\left(\frac{1+\sqrt{5}}{2}\right)^{n}.\]
We have also the following property : 
\begin{prop}
Denote by $q_{n}(x)$ the sequence of denominators from the continued
fraction algorithm of the number $x$; for all $\theta\in\mathbb{R}$,
$n\geq0$ we have \[
q_{n}(\theta)\geq\mathbb{F}_{n}.\]

\end{prop}
In the general case for any $\theta$ irrational number, we have the
following theorem (see for example \textbf{{[}Khi{]}}) :
\begin{thm}
\textbf{(Khintchine-Lévy, 1952)}. Almost surely for $\theta\in\mathbb{R}$
we have\[
\lim_{n\rightarrow+\infty}q_{n}(\theta)^{\frac{1}{n}}=K;\]
where $K$ denotes the Khintchine-Lévy constant $K:=e^{\frac{\pi}{12\ln(2)}}>1.$
\end{thm}
Thus for instance from a certain point we obtain :\[
\left(\frac{1}{2}K\right)^{n}\leq q_{n}(\theta)\leq\left(\frac{3}{2}K\right)^{n}.\]
The study of the distribution of the geometrical sequences $\left(\left(\frac{1}{2}K\right)^{n}\right)_{n\in\mathbb{N}}$
and $\left(\left(\frac{3}{2}K\right)^{n}\right)_{n\in\mathbb{N}}$
on the compact set $\left[h^{\min\delta_{i}^{\prime}-1-\mu},h^{1-2\min\delta_{i}+\mu}\right]$
is easy; unfortunately it does not provide accurate informations on
the distribution of the sequence $\left(q_{n}(\theta)\right)_{n\in\mathbb{N}}$
.

\subsubsection*{Open question}

Do we know the denominator distribution of $\left(q_{n}(\theta)\right)_{n\in\mathbb{N}}$
with the real axis depending on $\theta$? More specifically, for
a non-empty compact set of diameter $\delta>0$ includes in $\mathbb{R}_{+}^{*}$
is it possible to estimate the number of elements of the sequence
$\left(q_{n}(\theta)\right)_{n\in\mathbb{N}}$ in this compact set
depending on the numbers $\theta$ and $\delta$ ?

\section{Second order approximation : revival periods}

\subsection{Introduction }

Our next aim is to use a more accurate approximation of the function
$t\mapsto\mathbf{a}(t)$. In this section, we use the quadradic approximation
$\mathbf{a_{2}}(t)$ of the autocorrelation function, valid up on
a time scale $\left[0,1/h^{\beta}\right]$ where $\beta>1$. This
approximation is a consequence of a Taylor formula on the the term
$tF\left(\tau_{n},\mu_{m}\right)/h$ in order 2. In this quadradic
approximation appear three revivals periods $T_{rev_{1}},T_{rev_{2}}$
and $T_{rev_{12}}$ (of order $1/h)$. 
\selectlanguage{english}%
\begin{assumption}
\textup{In this section, we suppose $\frac{\partial^{2}F}{\partial X^{2}}\left(E_{1},E_{2}\right)\neq0,\,\frac{\partial^{2}F}{\partial X\partial Y}\left(E_{1},E_{2}\right)\neq0,\,\frac{\partial^{2}F}{\partial Y^{2}}\left(E_{1},E_{2}\right)\neq0.$}
\end{assumption}
\selectlanguage{english}%

\subsection{Quadradic approximation andc revival periods}

\selectlanguage{english}%

\subsubsection{Semi-classical revival and revival periods}
\selectlanguage{english}%
\begin{defn}
Let us define the semi-classical revival periods $T_{srev_{1}},$\foreignlanguage{english}{
$T_{srev_{2}}$ and }$T_{srev_{12}}$ by :\foreignlanguage{english}{\[
T_{srev{}_{1}}:=\frac{4\pi}{h\frac{\partial^{2}F}{\partial X^{2}}\left(\tau_{n_{0}},\mu_{m_{0}}\right)\omega_{1}^{2}};\, T_{srev_{2}}:=\frac{4\pi}{h\frac{\partial^{2}F}{\partial Y^{2}}\left(\tau_{n_{0}},\mu_{m_{0}}\right)\omega_{2}^{2}};\]
\[
T_{srev{}_{12}}:=\frac{4\pi}{h\frac{\partial^{2}F}{\partial X\partial Y}\left(\tau_{n_{0}},\mu_{m_{0}}\right)\omega_{1}\omega_{2}}.\]
}
\end{defn}
So we get the approximation :
\begin{prop}
Let $\beta$ a real number such that $\beta>1-3\min\delta_{i}$. Then
we have uniformly for $t\in\left[0,h^{\beta}\right]$:\[
e^{+itF\left(\tau_{n_{0}},\mu_{m_{0}}\right)/h}\mathbf{r}(t)\]
\[
={\displaystyle \sum_{n,m\in\mathbb{N}^{2}}\left|a_{n,m}\right|^{2}e^{-2i\pi t\left(\frac{n-n_{0}}{T_{scl_{1}}}+\frac{m-m_{0}}{T_{scl_{2}}}+\frac{(n-n_{0})^{2}}{T_{srev_{1}}}+\frac{(m-m_{0})^{2}}{T_{srev_{2}}}+\frac{(n-n_{0})(m-m_{0)}}{T_{srev_{12}}}\right)}}\]
\[
+O\left(h^{\beta+3\min\delta_{i}-1}\right).\]
\end{prop}
\begin{proof}
The principle is the same as in the proof of proposition 3.3. Here
we use the Taylor-Lagrange formula at order 3 : for all pair $(n,m)\in\mathbb{N}^{2}$
there exists $\theta=\theta\left(n,m,n_{0},m_{0}\right)\in\left]0,1\right[$
such that\[
F\left(\tau_{n},\mu_{m}\right)=F\left(\tau_{n_{0}},\mu_{m_{0}}\right)+\frac{\partial F\left(\tau_{n_{0}},\mu_{m_{0}}\right)}{\partial X}\omega_{1}h\left(n-n_{0}\right)+\frac{\partial F\left(\tau_{n_{0}},\mu_{m_{0}}\right)}{\partial y}\omega_{2}h\left(m-m_{0}\right)\]
\[
+\frac{1}{2}\frac{\partial^{2}F\left(\tau_{n_{0}},\mu_{m_{0}}\right)}{\partial X^{2}}\omega_{1}^{2}h^{2}\left(n-n_{0}\right)^{2}+\frac{1}{2}\frac{\partial^{2}F\left(\tau_{n_{0}},\mu_{m_{0}}\right)}{\partial Y^{2}}\omega_{2}^{2}h^{2}\left(m-m_{0}\right)^{2}\]

\[
+\frac{1}{2}\frac{\partial^{2}F}{\partial X\partial Y}\left(\tau_{n_{0}},\mu_{m_{0}}\right)\omega_{1}\omega_{2}h^{2}\left(n-n_{0}\right)\left(m-m_{0}\right)+\frac{1}{6}\frac{\partial^{3}F\left(\rho_{n,m}\right)}{\partial X^{3}}\omega_{1}^{3}h^{3}\left(n-n_{0}\right)^{3}\]
\[
+\frac{1}{2}\frac{\partial^{3}F\left(\rho_{n,m}\right)}{\partial X^{2}\partial Y}\omega_{1}^{2}\omega_{2}h^{3}\left(n-n_{0}\right)^{2}\left(m-m_{0}\right)+\frac{1}{2}\frac{\partial^{3}F\left(\rho_{n,m}\right)}{\partial X\partial Y^{2}}\omega_{1}\omega_{2}^{2}h^{3}\left(n-n_{0}\right)\left(m-m_{0}\right)^{2}\]
\[
+\frac{1}{6}\frac{\partial^{3}F\left(\rho_{n,m}\right)}{\partial X^{3}}\omega_{2}^{3}h^{3}\left(m-m_{0}\right)^{3},\]
with $\rho_{n,m}=\rho(n,m,n_{0},m_{0},h):=\left(\tau_{n_{0}}+\theta(\tau_{n}-\tau_{n_{0}}),\mu_{m_{0}}+\theta(\mu_{m}-\mu_{m_{0}})\right).$

Next, we observe that for all pair $(n,m)\in\Delta$ and for all $t\in[0,h^{\beta}]$\[
\left|t\left(n-n_{0}\right)^{3}h^{2}\right|\leq h^{\beta+3\delta_{1}-1};\;\left|t\left(n-n_{0}\right)^{2}\left(m-m_{0}\right)h^{2}\right|\leq h^{\beta+2\delta_{1}+\delta_{2}-1};\]
\[
\left|t\left(n-n_{0}\right)\left(m-m_{0}\right)^{2}h^{2}\right|\leq h^{\beta+\delta_{1}+2\delta_{2}-1};\;\left|t\left(m-m_{0}\right)^{3}h^{2}\right|\leq h^{\beta+3\delta_{2}-1};\]
hence, since $\beta>1-\min\left(3\delta_{1},2\delta_{1}+\delta_{2},\delta_{1}+2\delta_{2},3\delta_{2}\right)=1-3\min\delta_{i},$
for all $t\in\left[0,h^{\beta}\right]$ and for all pair $(n,m)\in\mathbb{N}^{2}$
we get \[
e^{-2i\pi t\left(\left(n-n_{0}\right)^{3}h^{2}+\left(n-n_{0}\right)^{2}\left(m-m_{0}\right)h^{2}+\left(n-n_{0}\right)\left(m-m_{0}\right)^{2}h^{2}+\left(m-m_{0}\right)^{3}h^{2}\right)}\]
\[
=1+O\left(h^{\beta-1+3\min\delta_{i}}\right).\]
And the statement of the proposition is established.
\end{proof}
For the same reason as in definition 3.4 we introduce the revival
periods :
\begin{defn}
Let us defines the revival periods $T_{rev_{1}},$\foreignlanguage{english}{
$T_{rev_{2}}$ and }$T_{rev_{12}}$\foreignlanguage{english}{ by :\[
T_{rev_{1}}:=\frac{4\pi}{h\frac{\partial^{2}F}{\partial X^{2}}\left(E_{1},E_{2}\right)\omega_{1}^{2}};\, T_{rev_{2}}:=\frac{4\pi}{h\frac{\partial^{2}F}{\partial Y^{2}}\left(E_{1},E_{2}\right)\omega_{2}^{2}};\]
\[
T_{rev{}_{12}}:=\frac{4\pi}{h\frac{\partial^{2}F}{\partial X\partial Y}\left(E_{1},E_{2}\right)\omega_{1}\omega_{2}}.\]
}
\end{defn}
Clearly for all $j\in\left\{ 1,2,12\right\} $ we have\foreignlanguage{english}{
$\lim_{h\rightarrow0}T_{srev_{j}}/T_{rev_{j}}=1.$} The three semi-classical
periods $T_{srev_{j}}$ depend on $h$ as well as their quotients.
Since we will consider period quotients afterwards, is it preferably
to study revival periods than semi-classical revival periods;for that
we use indeed :
\begin{prop}
Let $\upsilon$ a real number such that $\upsilon>-2\min\delta_{i}$.
Then we have uniformly for $t\in\left[0,h^{\upsilon}\right]$: \[
{\displaystyle {\displaystyle \sum_{n,m\in\mathbb{N}^{2}}\left|a_{n,m}\right|^{2}e^{-2i\pi t\left(\frac{n-n_{0}}{T_{scl_{1}}}+\frac{m-m_{0}}{T_{scl_{2}}}+\frac{(n-n_{0})^{2}}{T_{srev_{1}}}+\frac{(m-m_{0})^{2}}{T_{srev_{2}}}+\frac{(n-n_{0})(m-m_{0)}}{T_{srev_{12}}}\right)}}}\]
\[
={\displaystyle {\displaystyle \sum_{n,m\in\mathbb{N}^{2}}\left|a_{n,m}\right|^{2}e^{-2i\pi t\left(\frac{n-n_{0}}{T_{scl_{1}}}+\frac{m-m_{0}}{T_{scl_{2}}}+\frac{(n-n_{0})^{2}}{T_{rev_{1}}}+\frac{(m-m_{0})^{2}}{T_{rev_{2}}}+\frac{(n-n_{0})(m-m_{0)}}{T_{rev_{12}}}\right)}}}+O\left(h^{\upsilon+\min\left(\delta_{1},\delta_{2}\right)}\right).\]
\end{prop}
\begin{proof}
The principle is the same as in the proof of proposition 3.5. With
the partition $\mathbb{N}^{2}=\Delta\amalg\Gamma$ and by triangular
inequality we have \[
\left|\sum_{n,m\in\mathbb{N}^{2}}\left|a_{n,m}\right|^{2}\left[e^{-2i\pi t\left(\frac{n-n_{0}}{T_{scl_{1}}}+\frac{m-m_{0}}{T_{scl_{2}}}+\frac{(n-n_{0})^{2}}{T_{srev_{1}}}+\frac{(m-m_{0})^{2}}{T_{srev_{2}}}+\frac{(n-n_{0})(m-m_{0)}}{T_{srev_{12}}}\right)}\right.\right.\]
\[
\left.\left.e^{-2i\pi t\left(\frac{n-n_{0}}{T_{scl_{1}}}+\frac{m-m_{0}}{T_{scl_{2}}}+\frac{(n-n_{0})^{2}}{T_{rev_{1}}}+\frac{(m-m_{0})^{2}}{T_{rev_{2}}}+\frac{(n-n_{0})(m-m_{0)}}{T_{rev12}}\right)}\right]\right|\]
\[
\leq\sum_{n,m\in\Gamma}2\left|a_{n,m}\right|^{2}\]
\[
+2\sum_{n,m\in\Delta}\left|a_{n,m}\right|^{2}\left[\left|2\pi t(n-n_{0})^{2}\left(\frac{1}{T_{srev_{1}}}-\frac{1}{T_{rev_{1}}}\right)\right|+\left|2\pi t(m-m_{0})^{2}\left(\frac{1}{T_{srev_{2}}}-\frac{1}{T_{rev_{2}}}\right)\right|\right.\]
\[
+\left|2\pi t(n-n_{0})(m-m_{0})\left(\frac{1}{T_{srev_{12}}}-\frac{1}{T_{rev_{12}}}\right)\right|;\]
because $\left|e^{iX}-e^{iY}\right|\leq2\left|X-Y\right|.$ 

Next we observe that\[
T_{srev_{1}}-T_{rev_{1}}=\frac{4\pi}{\omega_{1}^{2}h}\left(\frac{\frac{\partial^{2}F}{\partial X^{2}}\left(E_{1},E_{2}\right)-\frac{\partial^{2}F}{\partial X^{2}}\left(\tau_{n_{0}},\mu_{m_{0}}\right)}{\frac{\partial^{2}F}{\partial X^{2}}\left(\tau_{n_{0}},\mu_{m_{0}}\right)\frac{\partial^{2}F}{\partial X^{2}}\left(E_{1},E_{2}\right)}\right).\]
First we have\[
\left|\frac{\partial^{2}F}{\partial X^{2}}\left(E_{1},E_{2}\right)-\frac{\partial^{2}F}{\partial X^{2}}\left(\tau_{n_{0}},\mu_{m_{0}}\right)\right|\]
\[
\leq\sup_{(x,y)\in B\left((E_{1},E_{2}),1\right)}\left\Vert \nabla\left(\frac{\partial^{2}F}{\partial X^{2}}\right)(x,y)\right\Vert _{\mathbb{R}^{2}}\left\Vert \left(E_{1},E_{2}\right)-\left(\tau_{n_{0}},\mu_{m_{0}}\right)\right\Vert _{\mathbb{R}^{2}}\]
\[
\leq M\sqrt{\left(E_{1}-\tau_{n_{0}}\right)^{2}+\left(E_{2}-\mu_{m_{0}}\right)^{2}}\leq Mh\frac{\sqrt{2}}{2};\]
where $M>0$ is a constant which does not depend on $h.$ 

On the other hand, since we suppose $\frac{\partial^{2}F}{\partial X^{2}}\left(E_{1},E_{2}\right)\neq0;$
there exists $\varepsilon_{1}>0$ and $r_{1}>0$ such that for all$(x,y)\in B\left((E_{1},E_{2}),r_{1}\right)$
we get \foreignlanguage{english}{\[
\left|\frac{\partial^{2}F}{\partial X^{2}}\left(x,y\right)\right|\geq\varepsilon_{1};\]
and we have seen that there exists $h_{1}>0$ such that for all $h\in\left]0,h_{1}\right[$
we have }\[
(\tau_{n_{0}},\mu_{m_{0}})\in B\left((E_{1},E_{2}),r_{1}\right);\]
therefore the application \[
h\mapsto\frac{1}{\frac{\partial^{2}F}{\partial X^{2}}\left(E_{1},E_{2}\right)\frac{\partial^{2}F}{\partial X^{2}}\left(\tau_{n_{0}},\mu_{m_{0}}\right)}\]
is bounded on the open set \foreignlanguage{english}{$\left]0,h_{1}\right[$,
indeed for all $h\in\left]0,h_{1}\right[$ we have \[
\left|\frac{1}{\frac{\partial^{2}F}{\partial X^{2}}\left(E_{1},E_{2}\right)\frac{\partial^{2}F}{\partial X^{2}}\left(\tau_{n_{0}},\mu_{m_{0}}\right)}\right|\leq\frac{1}{\varepsilon_{1}^{2}}<+\infty;\]
and with $M^{\prime}:=2\pi\frac{M\sqrt{2}}{\varepsilon_{1}^{2}}$
for all $h\in\left]0,h_{1}\right[$ we obtain $\left|T_{srev_{1}}-T_{rev_{1}}\right|\leq M^{\prime}.$} 

Next, since\[
\left|\frac{1}{T_{srev_{1}}T_{rev_{1}}}\right|\leq h^{2}\frac{\left|\frac{\partial^{2}F}{\partial X^{2}}\left(E_{1},E_{2}\right)\frac{\partial^{2}F}{\partial X^{2}}\left(\tau_{n_{0}},\mu_{m_{0}}\right)\right|}{16\pi^{2}}\leq Kh^{2}\]
where $K:=\frac{1}{16\pi^{2}}\sup_{(x,y)\in B\left((E_{1},E_{2}),1\right)}\left|\frac{\partial^{2}F}{\partial X^{2}}(x,y)\right|^{2}$,
there exists a constant $C_{1}>0$ (which does not depend on $h$
) such that for all $h\in\left]0,h_{1}\right[$ we have $\left|1/T_{sren_{1}}-1/T_{ren_{1}}\right|\leq C_{1}h^{2}$
(eg. take $C_{1}:=KM$). In a similary way : there exists $C_{2},C_{12}>0$
such that for all $h\in\left]0,h_{2}\right[$ we get $\left|1/T_{sren_{2}}-1/T_{ren_{2}}\right|\leq C_{2}h^{2}$
and for all $h\in\left]0,h_{12}\right[$ we get also $\left|1/T_{sren_{12}}-1/T_{ren_{12}}\right|\leq C_{12}h^{2}.$ 

Next, for all $t\geq0$, for all pair $(n,m)\in\Delta$ and for all
$h<\min\left(h_{1},h_{2},h_{12}\right)$ we have\[
\left|t\left(n-n_{0}\right)^{2}\left(\frac{1}{T_{srev_{1}}}-\frac{1}{T_{rev_{1}}}\right)\right|\leq C_{1}|t|h^{2\delta_{1}};\]
\[
\left|t\left(n-n_{0}\right)\left(m-m_{0}\right)\left(\frac{1}{T_{srev_{12}}}-\frac{1}{T_{rev_{12}}}\right)\right|\leq C_{12}|t|h^{\delta_{1}+\delta_{2}};\]
\[
\left|t\left(m-m_{0}\right)^{2}\left(\frac{1}{T_{srev_{2}}}-\frac{1}{T_{rev_{2}}}\right)\right|\leq C_{2}|t|h^{2\delta_{2}};\]
hence for all $t\in\left[0,h^{\upsilon}\right]$ where $\upsilon$
is a real number such that $\upsilon>-2\min\delta_{i}$ we obtain
\[
\left|t\left(n-n_{0}\right)^{2}\left(\frac{1}{T_{srev_{1}}}-\frac{1}{T_{rev_{1}}}\right)\right|+\left|t\left(m-m_{0}\right)^{2}\left(\frac{1}{T_{srev_{2}}}-\frac{1}{T_{rev_{2}}}\right)\right|\]
\[
+\left|t\left(n-n_{0}\right)\left(m-m_{0}\right)\left(\frac{1}{T_{srev_{12}}}-\frac{1}{T_{rev_{12}}}\right)\right|\leq Mh^{\upsilon+2\min\delta_{i}}\]
where $M:=3\max\left(C_{1},C_{2},C_{12}\right).$ 

So we proove that : there exists a constant $M>0$ such that for all
$h<\min\left(h_{1},h_{2},h_{12}\right)$ and for all $\upsilon>-2\min\delta_{i}$
we get \[
\left|\sum_{n,m\in\mathbb{N}^{2}}\left|a_{n,m}\right|^{2}\left[e^{-2i\pi t\left(\frac{n-n_{0}}{T_{scl_{1}}}+\frac{m-m_{0}}{T_{scl_{2}}}+\frac{(n-n_{0})^{2}}{T_{srev_{1}}}+\frac{(m-m_{0})^{2}}{T_{srev_{2}}}+\frac{(n-n_{0})(m-m_{0)}}{T_{srev_{12}}}\right)}\right.\right.\]
\[
\left.\left.e^{-2i\pi t\left(\frac{n-n_{0}}{T_{scl_{1}}}+\frac{m-m_{0}}{T_{scl_{2}}}+\frac{(n-n_{0})^{2}}{T_{rev_{1}}}+\frac{(m-m_{0})^{2}}{T_{rev_{2}}}+\frac{(n-n_{0})(m-m_{0)}}{T_{rev_{12}}}\right)}\right]\right|\]
\[
\leq2\sum_{n,m\in\Gamma}\left|a_{n,m}\right|^{2}+Mh^{\upsilon+2\min\delta_{i}}\sum_{n,m\in\Delta}\left|a_{n,m}\right|^{2}.\]
It follows from the lemma 2.5 (the lemma 2.5 says $\sum_{n,m\in\Gamma}\left|a_{n,m}\right|^{2}=O\left(h^{\infty}\right)$)
and from \[
\sum_{n,m\in\Delta}\left|a_{n,m}\right|^{2}\leq\sum_{n,m\in\mathbb{N}^{2}}\left|a_{n,m}\right|^{2}=1+O\left(h^{\infty}\right)\]
that \[
\left|\sum_{n,m\in\mathbb{N}^{2}}\left|a_{n,m}\right|^{2}\left[e^{-2i\pi t\left(\frac{n-n_{0}}{T_{scl_{1}}}+\frac{m-m_{0}}{T_{scl_{2}}}+\frac{(n-n_{0})^{2}}{T_{srev_{1}}}+\frac{(m-m_{0})^{2}}{T_{srev_{2}}}+\frac{(n-n_{0})(m-m_{0)}}{T_{srev_{12}}}\right)}\right.\right.\]
\[
\left.\left.e^{-2i\pi t\left(\frac{n-n_{0}}{T_{scl_{1}}}+\frac{m-m_{0}}{T_{scl_{2}}}+\frac{(n-n_{0})^{2}}{T_{rev_{1}}}+\frac{(m-m_{0})^{2}}{T_{rev_{2}}}+\frac{(n-n_{0})(m-m_{0)}}{T_{rev_{12}}}\right)}\right]\right|\]
\[
=O\left(h^{\upsilon+2\min\delta_{i}}\right).\]

\end{proof}

\subsubsection{Comparison between revivals periods an the time scale \foreignlanguage{english}{\textmd{\textup{$\left[0,h^{\alpha}\right]$}}}}

Recall here that the parameters $\left(\delta_{i}^{'},\delta_{i}\right)\in]\frac{1}{2},1[^{2}$
with $\delta_{i}^{'}>\delta_{i}$; recall also that the real coefficients
$\alpha$,$\beta$ would require $\alpha>1-2\min\delta_{i}$ and $\beta>1-3\min\delta_{i}$.
Next we observe that $1-3\min\delta_{i}-\left(1-2\min\delta_{i}\right)=-\min\delta_{i}<0$
and $-2\min\delta_{i}-\left(1-3\min\delta_{i}\right)=\min\delta_{i}-1<0$;
hence for $h$ small enough we obtain 

\[
h^{1-2\min\delta_{i}}<h^{1-3\min\delta_{i}}<h^{-2\min\delta_{i}}.\]
So we can make a {}``good choice'' for parameters $\alpha$, $\beta$
and $\upsilon$ : indeed we can choose $\alpha$, $\beta$ and $\upsilon$
such that : $\upsilon\leq\beta<-1<\alpha<0$, therefore (for $h$
small enough) we have :\[
\left[0,T_{cl_{i}}\right]\subset\left[0,h^{\alpha}\right]\subset\left[0,T_{rev_{i}}\right]\subset\left[0,h^{\beta}\right]\subseteq\left[0,h^{\upsilon}\right].\]

\subsubsection{The quadradic approximation $\mathbf{a_{2}}$}

So, the quadradic approximation of the autocorrelation function on
the time scale $\left[0,h^{\beta}\right]$ is :
\begin{defn}
The quadradic approximation of the autocorrelation function is \[
\mathbf{\mathbf{a_{2}}\,:\,}t\mapsto{\displaystyle \sum_{n,m\in\mathbb{N}^{2}}\left|a_{n,m}\right|^{2}e^{-2i\pi t\left(\frac{n-n_{0}}{T_{scl_{1}}}+\frac{m-m_{0}}{T_{scl_{2}}}+\frac{(n-n_{0})^{2}}{T_{rev_{1}}}+\frac{(m-m_{0})^{2}}{T_{rev_{2}}}+\frac{(n-n_{0})(m-m_{0)}}{T_{rev_{12}}}\right)}.}\]

\end{defn}

\subsection{Revival theorems}

\subsubsection{Preliminaries }

\subsection*{Resonance hypothesis}
\begin{defn}
We say that the revival periods $T_{rev_{1}},\, T_{rev_{2}},\, T_{rev_{12}}$
are in resonance if an only if there exists $\left(\frac{p_{1}}{q_{1}},\frac{p_{2}}{q_{2}},\frac{p_{12}}{q_{12}}\right)\in\mathbb{Q}^{3}$
such that\[
\frac{p_{1}}{q_{1}}T_{rev_{1}}=\frac{p_{2}}{q_{2}}T_{rev_{2}}=\frac{p_{12}}{q_{12}}T_{rev_{12}}.\]
\end{defn}
\begin{notation}
In this case, we introduce the notation $T_{frac}:=\frac{p_{1}}{q_{1}}T_{rev_{1}}=\frac{p_{2}}{q_{2}}T_{rev_{2}}=\frac{p_{12}}{q_{12}}T_{rev_{12}}.$
And for all $j\in\{1,2\}$ let us also consider the numbers $r_{j}:=p_{12}q_{j},\; s_{j}:=q_{12}p_{j}$
, and clearly for all $j\in\{1,2\}$ we have $T_{rev_{j}}=\frac{r_{j}}{s_{j}}T_{rev_{12}}.$
\end{notation}

\subsection*{Preliminaries }

To make progress in our study we need to introduce a new function
$\psi_{cl}$ with two artificial variables $t_{1},t_{2}$.
\begin{defn}
Let us define the pseudo-classical function $\psi_{cl}$ :\[
\psi_{cl}\left(t_{1},t_{2}\right):={\displaystyle \sum_{n,m\in\mathbb{N}^{2}}\left|a_{n,m}\right|^{2}e^{-2i\pi t_{1}\frac{n-n_{0}}{T_{scl_{1}}}-2i\pi t_{2}\frac{m-m_{0}}{T_{scl_{2}}}}.}\]

\end{defn}
So we get the obvious following property; first the function $\psi_{cl}$
is doubly-periodic :

\textbf{(i)} for all pair $t_{1},t_{2}\geq0$ we have $\psi_{cl}\left(t_{1}+T_{scl_{1}},t_{2}\right)={\displaystyle \psi_{cl}\left(t_{1},t_{2}\right);}$

\textbf{(ii)} and for all pair $t_{1},t_{2}\geq0$ we have also $\psi_{cl}\left(t_{1},t_{2}+T_{scl_{2}}\right)={\displaystyle \psi_{cl}\left(t_{1},t_{2}\right).}$

This function have no immediate physical significance, but if the
time $t_{1}$ and $t_{2}$ are equal :

\textbf{(iii)} for all $t\geq0$ we have $\psi_{cl}(t,t)={\displaystyle \mathbf{a}_{\mathbf{1}}(t).}$

\subsection*{Some lemmas}
\begin{notation}
Let us consider the sequence $\left(\theta_{n,m}\right)_{n,m}=\left(\theta_{n,m}(p_{1},q_{1},p_{2},q_{2},p_{12},q_{12},h)\right)_{n,m}$
with $(n,m)\in\mathbb{Z}^{2}$ defined by :\[
\theta_{n,m}:=e^{-2i\pi\left(\frac{p_{1}}{q_{1}}(n-n_{0})^{2}+\frac{p_{12}}{q_{12}}(n-n_{0})(m-m_{0})+\frac{p_{2}}{q_{2}}(m-m_{0})^{2}\right)}.\]

\end{notation}
The periodicity of this sequence is caracterised by the following
easy proposition. 
\begin{prop}
For all $p_{1},q_{1},p_{2},q_{2},p_{12},q_{12}\in\mathbb{Z}$, the
sequence $\left(\theta_{n,m}\right)_{n,m}$ verify \[
\left\{ \begin{array}{cc}
\theta_{n+\ell_{1},m}=\theta_{n,m}\\
\\\theta_{n,m+\ell_{2}}=\theta_{n,m}\end{array}\right.\]
if and only if the integers $\ell_{1}$ and $\ell_{2}$ satisfy the
following equations :\[
\forall(n,m)\in\mathbb{Z}^{2},\;\frac{\ell_{1}^{2}p_{1}}{q_{1}}+\frac{2p_{1}\ell_{1}}{q_{1}}n+\frac{p_{1}r_{1}\ell_{1}}{s_{1}q_{1}}m\equiv0\;[1]\]
\[
\forall(n,m)\in\mathbb{Z}^{2},\;\frac{\ell_{2}^{2}p_{2}}{q_{2}}+\frac{2p_{2}\ell_{2}}{q_{2}}m+\frac{p_{2}r_{2}\ell_{2}}{s_{2}q_{2}}n\equiv0\;[1].\]
\end{prop}
\begin{example}
An obvious solution is $\ell_{1}=q_{1}s_{1}$ and $\ell_{2}=q_{2}s_{2}.$
\end{example}
For two periods $\ell_{1},\ell_{2}\in\mathbb{Z}^{2}$ let us consider
the set of sequences $\ell_{1},\ell_{2}-$periodic with his natural
scalar product.
\begin{defn}
For a fixed pair $\ell_{1},\ell_{2}$$\in(\mathbb{Z}^{*})^{2}$ we
define $\mathfrak{S}_{\ell_{1},\ell_{2}}(\mathbb{Z})$ the set of
sequences $\ell_{1},\ell_{2}-$periodic in the following sense : \[
\mathfrak{S}_{\ell_{1},\ell_{2}}(\mathbb{Z}):=\left\{ u_{n,m}\in\mathbb{C}^{\mathbb{Z}^{2}};\,\forall n,m\in\mathbb{Z}^{2},\, u_{n+\ell_{1},m}=u_{n,m}\;\; and\;\; u_{n,m+\ell_{2}}=u_{n,m}\right\} .\]

\end{defn}
So we have the elementary :
\begin{prop}
The application\[
\left\langle \,,\,\right\rangle _{\mathfrak{S}_{\ell_{1},\ell_{2}}}:\left\{ \begin{array}{cc}
\mathfrak{S}_{\ell_{1},\ell_{2}}(\mathbb{Z})^{2}\rightarrow\mathbb{C}\\
\\(u,v)\mapsto\left\langle u,v\right\rangle _{\mathfrak{S}_{\ell_{1},\ell_{2}}}:=\frac{1}{|\ell_{1}\ell_{2}|}{\displaystyle \sum_{n=0}^{|\ell_{1}|-1}\sum_{m=0}^{|\ell_{2}|-1}u_{n,m}\overline{v_{n,m}}} & .\end{array}\right.\]
is a Hermitean product on the space $\mathfrak{S}_{\ell_{1},\ell_{2}}(\mathbb{Z})$
.
\end{prop}
We have also the obvious following remark :
\begin{prop}
Let us consider $\phi_{n,m}^{k,p}:=e^{-\frac{2i\pi kn}{\ell_{1}}}e^{-\frac{2i\pi pm}{\ell_{2}}}$
where $(k,p)\in\mathbb{Z}^{2}$; then the family $\left\{ \left(\phi_{n,m}^{k,p}\right)_{n,m\in\mathbb{Z}^{2}}\right\} _{k=0...\ell_{1}-1,p=0...\ell_{2}-1}$
is an orthonormal basis of the space vector $\mathfrak{S}_{\ell_{1},\ell_{2}}(\mathbb{Z})$.
\end{prop}

\subsubsection{The main theorem}

In the following theorem we show that the function $t\mapsto\mathbf{a}_{\mathbf{2}}(t)$
near the period \foreignlanguage{english}{$T_{frac}$} can be written
as a finite sum of $\psi_{cl}$ with arguments shifted. Indeed we
have :
\begin{thm}
Suppose resonance hypothesis holds; then there exists a family of
$\ell_{1}+\ell_{2}$ complex numbers (depends on $h$) : $\left(c_{k_{1},k_{2}}\right)_{k_{1}\in\left\{ 0...\ell_{1}-1\right\} ,k_{2}\in\left\{ 0...\mathbf{\ell}_{2}-1\right\} }$
where the integers $\ell_{1},\ell_{2}\in\mathbb{Z}^{2}$ are solutions
of equations from proposition 4.11; such that\[
\mathbf{a}_{\mathbf{2}}\left(t+T_{frac}\right)={\displaystyle \sum_{k_{1}=0}^{\ell_{1}-1}\sum_{k_{2}=0}^{\ell_{2}-1}c_{k_{1},k_{2}}\psi_{cl}\left(t+T_{frac}+\frac{k_{1}}{\ell_{1}}T_{scl_{1}},t+T_{frac}+\frac{k_{2}}{\ell_{2}}T_{scl_{2}}\right)}\]
\[
+O\left(h^{\alpha+2\min\delta_{i}-1}\right).\]
holds for all $t\in\left[0,h^{\alpha}\right]$. The numbers $c_{k_{1},k_{2}}$
are called fractionnals coefficients; and for all $k_{1}\in\left\{ 0...\ell_{1}-1\right\} ,\, k_{2}\in\left\{ 0...\ell_{2}-1\right\} $
\[
c_{k_{1},k_{2}}=e^{-\frac{2i\pi k_{1}n_{0}}{\ell_{1}}}e^{-\frac{2i\pi k_{2}m_{0}}{\ell_{2}}}b_{k_{1},k_{2}}\]
with $b_{k_{1},k_{2}}=b_{k_{1},k_{2}}(h)=\left\langle \sigma_{h},\phi^{k_{1},k_{2}}\right\rangle _{\mathfrak{S}_{\ell_{1},\ell_{2}}}.$\end{thm}
\begin{proof}
Let us denote the integers $\widetilde{n}:=n-n_{0},\,\widetilde{m}:=m-m_{0}$
and consider the function $\varepsilon(t)$ : \[
\varepsilon(t):=\left|\mathbf{a}_{\mathbf{2}}\left(t+T_{frac}\right)-{\displaystyle \sum_{k_{1}=0}^{\ell_{1}-1}\sum_{k_{2}=0}^{\ell_{2}-1}c_{k_{1},k_{2}}\psi_{cl}\left(t+T_{frac}+\frac{k_{1}}{\ell_{1}}T_{scl_{1}},t+T_{frac}+\frac{k_{2}}{\ell_{2}}T_{scl_{2}}\right)}\right|\]
\[
=\left|\sum_{n,m\in\mathbb{N}^{2}}\left|a_{n,m}\right|^{2}e^{-2i\pi t\frac{\widetilde{n}}{T_{scl_{1}}}}e^{-2i\pi T_{frac}\frac{\widetilde{n}}{T_{scl_{1}}}}e^{-2i\pi t\frac{\widetilde{m}}{T_{scl_{2}}}}e^{-2i\pi T_{frac}\frac{\widetilde{m}}{T_{scl_{2}}}}\right.\]
\[
e^{-2i\pi t\frac{\widetilde{n}^{2}}{T_{rev_{1}}}}e^{-2i\pi\frac{p_{1}\widetilde{n}^{2}}{q_{1}}}e^{-2i\pi t\frac{\widetilde{m}^{2}}{T_{rev_{2}}}}e^{-2i\pi\frac{p_{2}\widetilde{m}^{2}}{q_{2}}}e^{-2i\pi t\frac{\widetilde{n}\widetilde{m}}{T_{rev{}_{12}}}}e^{-2i\pi\frac{p_{12}\widetilde{n}\widetilde{m}}{q_{12}}}\]
\foreignlanguage{english}{\[
\left.-\sum_{k_{1}=0}^{\ell_{1}-1}\sum_{k_{2}=0}^{\ell_{2}-1}c_{k_{1},k_{2}}\psi_{cl}\left(t+T_{frac}+\frac{k_{1}}{\ell_{1}}T_{cl_{1}},t+T_{frac}+\frac{k_{2}}{\ell_{2}}T_{cl_{2}}\right)\right|\]
}\[
=\left|\sum_{n,m\in\mathbb{N}^{2}}\left|a_{n,m}\right|^{2}\theta_{n,m}e^{-2i\pi t\frac{\widetilde{n}}{T_{scl_{1}}}}e^{-2i\pi T_{frac}\frac{\widetilde{n}}{T_{scl_{1}}}}e^{-2i\pi t\frac{\widetilde{m}}{T_{scl_{2}}}}e^{-2i\pi T_{frac}\frac{\widetilde{m}}{T_{scl_{2}}}}\right.\]
\[
e^{-2i\pi t\frac{\widetilde{n}^{2}}{T_{rev_{1}}}}e^{-2i\pi t\frac{\widetilde{m}^{2}}{T_{rev{}_{2}}}}e^{-2i\pi t\frac{\widetilde{n}\widetilde{m}}{T_{rev_{12}}}}\]
\foreignlanguage{english}{\[
\left.-\sum_{k_{1}=0}^{\ell_{1}-1}\sum_{k_{2}=0}^{\ell_{2}-1}c_{k_{1},k_{2}}\psi_{cl}\left(t+T_{frac}+\frac{k_{1}}{\ell_{1}}T_{scl_{1}},t+T_{frac}+\frac{k_{2}}{\ell_{2}}T_{scl_{2}}\right)\right|.\]
}Since the sequence $\left(\theta_{n,m}\right)_{n,m}\in\mathfrak{S}_{\ell_{1},\ell_{2}}(\mathbb{Z})$
with $\ell_{1}=q_{1}s_{1}$ and $\ell_{2}=q_{2}s_{2}$ there exists
a unique decomposition of the sequence \foreignlanguage{english}{$\left(\theta_{n,m}\right)_{n,m}$}
on the basis \textit{$\left\{ \left(\phi_{n,m}^{k,p}\right)_{n,m\in\mathbb{Z}^{2}}\right\} _{k=0...\ell_{1}-1,p=0...\ell_{2}-1}$}
; indeed we have :\[
\theta_{n,m}=\sum_{k_{1}=0}^{\ell_{1}-1}{\displaystyle \sum_{k_{2}=0}^{\ell_{2}-1}b_{k_{1},k_{2}}\phi_{n,m}^{k_{1},k_{2}}}.\]
where $b_{k_{1},k_{2}}=\left\langle \theta,\phi^{k_{1},k_{2}}\right\rangle _{\mathfrak{S}_{\ell_{1},\ell_{2}}}$.
Therefore we get \[
\varepsilon(t)=\left|\sum_{n,m\in\mathbb{N}^{2}}\sum_{k_{1}=0}^{\ell_{1}-1}\sum_{k_{2}=0}^{\ell_{2}-1}\left|a_{n,m}\right|^{2}b_{k_{1},k_{2}}e^{-2i\pi t\frac{\widetilde{n}}{T_{scl_{1}}}}e^{-2i\pi T_{frac}\frac{\widetilde{n}}{T_{scl_{1}}}}e^{-2i\pi t\frac{\widetilde{m}}{T_{scl_{2}}}}e^{-2i\pi T_{frac}\frac{\widetilde{m}}{T_{scl_{2}}}}\right.\]
\[
e^{-2i\pi t\frac{\widetilde{n}^{2}}{T_{rev_{1}}}}e^{-2i\pi t\frac{\widetilde{m}^{2}}{T_{rev_{2}}}}e^{-2i\pi t\frac{\widetilde{n}\widetilde{m}}{T_{rev_{12}}}}\phi_{n,m}^{k_{1},k_{2}}\]
\foreignlanguage{english}{\[
\left.-\sum_{k_{1}=0}^{\ell_{1}-1}\sum_{k_{2}=0}^{\ell_{2}-1}c_{k_{1},k_{2}}\psi_{cl}\left(t+T_{frac}+\frac{k_{1}}{\ell_{1}}T_{scl_{1}},t+T_{frac}+\frac{k_{2}}{\ell_{2}}T_{scl_{2}}\right)\right|\]
}

\[
=\left|\sum_{n,m\in\mathbb{N}^{2}}\sum_{k_{1}=0}^{\ell_{1}-1}\sum_{k_{2}=0}^{\ell_{2}-1}\left|a_{n,m}\right|^{2}b_{k_{1},k_{2}}e^{-2i\pi t\frac{\widetilde{n}}{T_{scl_{1}}}}e^{-2i\pi T_{frac}\frac{\widetilde{n}}{T_{scl_{1}}}}e^{-2i\pi t\frac{\widetilde{m}}{T_{scl_{2}}}}e^{-2i\pi T_{frac}\frac{\widetilde{m}}{T_{scl_{2}}}}\right.\]
\[
e^{-2i\pi t\frac{\widetilde{n}^{2}}{T_{rev_{1}}}}e^{-2i\pi t\frac{\widetilde{m}^{2}}{T_{rev_{2}}}}e^{-2i\pi t\frac{\widetilde{n}\widetilde{m}}{T_{rev_{12}}}}\phi_{n,m}^{k_{1},k_{2}}\]
\foreignlanguage{english}{\[
-\sum_{k_{1}=0}^{\ell_{1}-1}\sum_{k_{2}=0}^{\ell_{2}-1}c_{k_{1},k_{2}}\sum_{n,m\in\mathbb{N}^{2}}\left|a_{n,m}\right|^{2}e^{-2i\pi t\frac{\widetilde{n}}{T_{scl_{1}}}}e^{-2i\pi T_{frac}\frac{\widetilde{n}}{T_{scl_{1}}}}e^{-2i\pi k_{1}\frac{\widetilde{n}}{\ell_{1}}}\]
\[
\left.e^{-2i\pi t\frac{\widetilde{m}}{T_{scl_{2}}}}e^{-2i\pi T_{frac}\frac{\widetilde{m}}{T_{scl_{2}}}}e^{-2i\pi k_{2}\frac{\widetilde{m}}{\ell_{2}}}\right|\]
}\[
=\left|\sum_{n,m\in\mathbb{N}^{2}}\sum_{k_{1}=0}^{\ell_{1}-1}\sum_{k_{2}=0}^{\ell_{2}-1}\left|a_{n,m}\right|^{2}b_{k_{1},k_{2}}e^{-2i\pi t\frac{\widetilde{n}}{T_{scl_{1}}}}e^{-2i\pi T_{frac}\frac{\widetilde{n}}{T_{scl_{1}}}}e^{-2i\pi t\frac{\widetilde{m}}{T_{scl_{2}}}}e^{-2i\pi T_{frac}\frac{\widetilde{m}}{T_{scl_{2}}}}\right.\]
\[
e^{-2i\pi t\frac{\widetilde{n}^{2}}{T_{rev_{1}}}}e^{-2i\pi t\frac{\widetilde{m}^{2}}{T_{rev_{2}}}}e^{-2i\pi t\frac{\widetilde{n}\widetilde{m}}{T_{rev_{12}}}}e^{-\frac{2i\pi k_{1}n}{\ell_{1}}}e^{-\frac{2i\pi k_{2}m}{\ell_{2}}}\]
\foreignlanguage{english}{\[
-\sum_{k_{1}=0}^{\ell_{1}-1}\sum_{k_{2}=0}^{\ell_{2}-1}c_{k_{1},k_{2}}\sum_{n,m\in\mathbb{N}^{2}}\left|a_{n,m}\right|^{2}e^{-2i\pi t\frac{\widetilde{n}}{T_{scl_{1}}}}e^{-2i\pi T_{frac}\frac{\widetilde{n}}{T_{scl_{1}}}}e^{-2i\pi k_{1}\frac{\widetilde{n}}{\ell_{1}}}\]
\[
\left.e^{-2i\pi t\frac{\widetilde{m}}{T_{scl_{2}}}}e^{-2i\pi T_{frac}\frac{\widetilde{m}}{T_{scl_{2}}}}e^{-2i\pi k_{2}\frac{\widetilde{m}}{\ell_{2}}}\right|.\]
}And since $c_{k_{1},k_{2}}e^{-2i\pi k_{1}\frac{\widetilde{n}}{\ell_{1}}}e^{-2i\pi k_{2}\frac{\widetilde{m}}{\ell_{2}}}=b_{k_{1},k_{2}}e^{-\frac{2i\pi k_{1}n}{\ell_{1}}}e^{-\frac{2i\pi k_{2}m}{\ell_{2}}}$
we deduce that 

\[
\varepsilon(t)=\]
\[
\left|\sum_{n,m\in\mathbb{N}^{2}}\sum_{k_{1}=0}^{\ell_{1}-1}\sum_{k_{2}=0}^{\ell_{2}-1}\left|a_{n,m}\right|^{2}b_{k_{1},k_{2}}e^{-2i\pi t\frac{\widetilde{n}}{T_{scl_{1}}}}e^{-2i\pi T_{frac}\frac{\widetilde{n}}{T_{scl_{1}}}}e^{-2i\pi t\frac{\widetilde{m}}{T_{scl_{2}}}}e^{-2i\pi T_{frac}\frac{\widetilde{m}}{T_{scl_{2}}}}e^{-\frac{2i\pi k_{1}n}{\ell_{1}}}e^{-\frac{2i\pi k_{2}m}{\ell_{2}}}\right.\]
\foreignlanguage{english}{\[
\left.\left(-1+e^{-2i\pi t\frac{\widetilde{n}^{2}}{T_{rev{}_{1}}}}e^{-2i\pi t\frac{\widetilde{m}^{2}}{T_{rev_{2}}}}e^{-2i\pi t\frac{\widetilde{n}\widetilde{m}}{T_{rev{}_{12}}}}\right)\right|.\]
}To finish, we use the partition $\mathbb{N}^{2}=\Delta\amalg\Gamma$
and we just consider indices in the set $\Delta$ for the sum. Hence
there exists constants $C_{1},C_{2},C_{12}>0$ which does not depend
on $h$, such that \[
\left|t\frac{(n-n_{0})^{2}}{T_{rev_{1}}}\right|\leq C_{1}h^{\alpha+2\delta_{1}-1};\;\left|t\frac{(m-m_{0})^{2}}{T_{rev_{2}}}\right|\leq C_{2}h^{\alpha+2\delta_{2}-1};\]
\[
\left|t\frac{(n-n_{0})(m-m_{0})}{T_{rev_{12}}}\right|\leq C_{12}h^{\alpha+\delta_{1}+\delta_{2}-1}\]
holds for all pair $(n,m)\in\Delta$ and for all \foreignlanguage{english}{$t\in\left[0,h^{\alpha}\right]$.}

As a consequence for all \foreignlanguage{english}{$t\in\left[0,h^{\alpha}\right]$
and for all pair }$(n,m)\in\Delta$ we get :\[
e^{-2i\pi t\frac{\widetilde{n}^{2}}{T_{rev_{1}}}}e^{-2i\pi t\frac{\widetilde{m}^{2}}{T_{rev_{2}}}}e^{-2i\pi t\frac{\widetilde{n}\widetilde{m}}{T_{rev_{12}}}}-1=O\left(h^{\alpha+2\min\delta_{i}-1}\right).\]

\end{proof}
If we take $t=0$ we obtain :
\begin{cor}
Under the same hypothesis we have\[
\mathbf{a}_{\mathbf{2}}\left(T_{frac}\right)={\displaystyle \sum_{k_{1}=0}^{\ell_{1}-1}\sum_{k_{2}=0}^{\ell_{2}-1}c_{k_{1},k_{2}}\psi_{cl}\left(T_{frac}+\frac{k_{1}}{\ell_{1}}T_{scl_{1}},T_{frac}+\frac{k_{2}}{\ell_{2}}T_{scl_{2}}\right)}.\]

\end{cor}

\subsection{Explicit values of modulus for revival coefficients}

Our final aim is to compute the modulus of revival coefficients. The
idea is to split the sum \foreignlanguage{english}{$\left|c_{k_{1},k_{2}}\right|$}
in two simple parts. These parts look like that Gauss sums, but in
fact with a little difference. Here we propose a simple way to compute
this sums and we don't use sophisticated theory. Start with a notation
and a remark :
\begin{notation}
For $\ell\geq1$ and for integers  $p$ et $q$ such
that $p\wedge q=1$, let us consider for all integer $k\in\left\{ 0...\ell-1\right\} $
the following sum \[
d_{k}(\ell,p,q)=\frac{1}{\ell}{\displaystyle \sum_{n=0}^{\ell-1}e^{-2i\pi\frac{p}{q}(n-n_{0})^{2}}e^{\frac{2i\pi kn}{\ell}}}.\]

\end{notation}
Therefore for all\textit{ $k\in\left\{ 0...\ell-1\right\} $ }\[
\left|d_{k}(\ell,p,q)\right|=\frac{1}{\ell}\left|{\displaystyle \sum_{m\in\mathbb{Z}/\ell\mathbb{Z}}e^{-2i\pi\frac{p}{q}m^{2}}e^{\frac{2i\pi km}{\ell}}}\right|.\]

\begin{thm}
Suppose resonance hypothesis holds and suppose also $\frac{p_{12}}{q_{12}}\in\mathbb{Z};$
then we obtain :\[
\left|b_{k_{1},k_{2}}\right|^{2}=\left|d_{k_{1}}\left(\ell_{1},p_{1}s_{1},\ell_{1}\right)\right|^{2}\left|d_{k_{2}}\left(\ell_{2},p_{2}s_{2},\ell_{2}\right)\right|^{2}.\]
\end{thm}
\begin{proof}
For all $k_{1}\in\left\{ 0...\ell_{1}-1\right\} ,\, k_{2}\in\left\{ 0...\ell_{2}-1\right\} $
we have \[
\left|b_{k_{1},k_{2}}\right|=\frac{1}{\ell_{1}\ell_{2}}\left|{\displaystyle \sum_{(n,m)\in\mathbb{Z}/\ell_{1}\mathbb{Z}\times\mathbb{Z}/\ell_{2}\mathbb{Z}}e^{-2i\pi\left(\frac{p_{1}}{q_{1}}n^{2}+\frac{p_{12}}{q_{12}}nm+\frac{p_{2}}{q_{2}}m^{2}\right)}e^{\frac{2i\pi k_{1}n}{\ell_{1}}}e^{\frac{2i\pi k_{2}m}{\ell_{2}}}}\right|;\]
\[
=\frac{1}{\ell_{1}\ell_{2}}\left|{\displaystyle \sum_{m=0}^{\ell_{2}-1}e^{\frac{2i\pi k_{2}m}{\ell_{2}}}e^{-2i\pi\frac{p_{2}}{q_{2}}m^{2}}\sum_{n=0}^{\ell_{1}-1}e^{-2i\pi\frac{p_{1}}{q_{1}}n^{2}}e^{\frac{2i\pi n}{\ell_{1}}\left(k_{1}-\frac{p_{12}}{q_{12}}\ell_{1}m\right)}}\right|\]
and since $\ell_{1}=q_{1}s_{1}=q_{1}q_{12}p_{1}$ we obtain \[
\left|b_{k_{1},k_{2}}\right|=\frac{1}{\ell_{1}\ell_{2}}\left|{\displaystyle \sum_{m=0}^{\ell_{2}-1}e^{\frac{2i\pi k_{2}m}{\ell_{2}}}e^{-2i\pi\frac{p_{2}}{q_{2}}m^{2}}\sum_{n=0}^{\ell_{1}-1}e^{-2i\pi\frac{p_{1}}{q_{1}}n^{2}}e^{\frac{2i\pi n}{\ell_{1}}\left(k_{1}-p_{12}q_{1}s_{1}m\right)}}\right|.\]
For $j\in\{1,2\},$ let us consider $\chi_{j}$ the following characters
: \[
\chi_{j}:\left\{ \begin{array}{cc}
\mathbb{Z}/\ell_{j}\mathbb{Z}\hookrightarrow\mathbb{C}^{*}\\
\\a\mapsto e^{-2i\pi\frac{a}{\ell_{j}}};\end{array}\right.\]
as a consequence we get\[
\left|b_{k_{1},k_{2}}\right|^{2}=\frac{1}{\ell_{1}^{2}\ell_{2}^{2}}{\displaystyle \sum_{(x,y)\in\mathbb{Z}/\ell_{1}\mathbb{Z}\times\mathbb{Z}/\ell_{2}\mathbb{Z}}\chi_{1}\left(p_{1}s_{1}x^{2}-x(k_{1}-p_{12}q_{1}p_{1}y\right)\chi_{2}\left(p_{2}s_{2}y^{2}-k_{2}y\right)}\]
\foreignlanguage{english}{\[
{\displaystyle \sum_{(z,t)\in\mathbb{Z}/\ell_{1}\mathbb{Z}\times\mathbb{Z}/\ell_{2}\mathbb{Z}}\chi_{1}\left(-p_{1}s_{1}z^{2}+z(k_{1}-p_{12}q_{1}p_{1}t\right)\chi_{2}\left(-p_{2}s_{2}t^{2}+k_{2}t\right)}\]
}\[
=\frac{1}{\ell_{1}^{2}\ell_{2}^{2}}\sum_{\left((x,z),(y,t)\right)\in(\mathbb{Z}/\ell_{1}\mathbb{Z})^{2}\times(\mathbb{Z}/\ell_{2}\mathbb{Z})^{2}}\left(\chi_{1}\left(p_{1}s_{1}\left(x^{2}-z^{2}\right)-k_{1}(x-z)+p_{12}q_{1}p_{1}(xy-zt)\right)\right.\]
\[
\left.{\displaystyle \chi_{2}\left(p_{2}s_{2}\left(y^{2}-t^{2}\right)-k_{2}(y-t)\right)}\right).\]
Now, because $\frac{p_{12}}{q_{12}}\in\mathbb{Z}$ then $\ell_{1}=q_{1}q_{12}p_{1}|q_{1}p_{1}p_{12}$
hence for all $x,y,z,t$ we have $p_{12}q_{1}p_{1}(xy-zt)\in\ell_{1}\mathbb{Z}$
. Therefore for all $x,y,z,t$ we have \foreignlanguage{english}{$\chi_{1}\left(p_{12}q_{1}p_{1}(xy-zt)\right)=1.$}
Hence\[
\left|b_{k_{1},k_{2}}\right|^{2}=\]
\[
\frac{1}{\ell_{1}^{2}\ell_{2}^{2}}\sum_{\left((x,z),(y,t)\right)\in(\mathbb{Z}/\ell_{1}\mathbb{Z})^{2}\times(\mathbb{Z}/\ell_{2}\mathbb{Z})^{2}}\chi_{1}\left(p_{1}s_{1}\left(x^{2}-z^{2}\right)-k_{1}(x-z)\right){\displaystyle \chi_{2}\left(p_{2}s_{2}\left(y^{2}-t^{2}\right)-k_{2}(y-t)\right)}\]
\[
=\frac{1}{\ell_{1}^{2}\ell_{2}^{2}}\sum_{(x,z)\in(\mathbb{Z}/\ell_{1}\mathbb{Z})^{2}}\chi_{1}\left(p_{1}s_{1}\left(x^{2}-z^{2}\right)-k_{1}(x-z)\right)\sum_{(y,t)\in(\mathbb{Z}/\ell_{2}\mathbb{Z})^{2}}\chi_{2}\left(p_{2}s_{2}\left(y^{2}-t^{2}\right)-k_{2}(y-t)\right)\]
\foreignlanguage{english}{\[
=\frac{1}{\ell_{1}^{2}}\left|{\displaystyle \sum_{x\in\mathbb{Z}/\ell_{1}\mathbb{Z}}e^{-2i\pi\frac{p_{1}s_{1}}{\ell_{1}}x^{2}}e^{\frac{2i\pi k_{1}}{\ell_{1}}x}}\right|^{2}\frac{1}{\ell_{2}^{2}}\left|{\displaystyle \sum_{y\in\mathbb{Z}/\ell_{2}\mathbb{Z}}e^{-2i\pi\frac{p_{2}s_{2}}{\ell_{2}}y^{2}}e^{\frac{2i\pi k_{2}}{\ell_{2}}y}}\right|^{2}.\]
}
\end{proof}
To finish we can compute \[
\left|d_{k_{1}}\left(\ell_{1},p_{1}s_{1},\ell_{1}\right)\right|^{2}\left|d_{k_{2}}\left(\ell_{2},p_{2}s_{2},\ell_{2}\right)\right|^{2}.\]
with the following results (see for example\textbf{ {[}Lab2{]}}) :
\begin{prop}
For all pair $p,q$ with $p\wedge q=1$ and $q$ odd, then for all
$k\in\left\{ 0...q-1\right\} $ we get :\[
\left|d_{k}(q,p,q)\right|^{2}=\frac{1}{q}.\]

\end{prop}
And
\begin{prop}
For all pair $p,q$ with $p\wedge q=1$ and $q$ even, then for all
$k\in\left\{ 0...q-1\right\} $ we get :\[
if\;\frac{q}{2}\; is\; even\; then\;\left|d_{k}(q,p,q)\right|^{2}=\left\{ \begin{array}{cc}
\frac{2}{q}\;\; if\; k\; is\; even\\
\\0\;\; else;\end{array}\right.\]
\[
if\:\frac{q}{2}\; is\: odd\; then\;\left|d_{k}(q,p,q)\right|^{2}=\left\{ \begin{array}{cc}
0\;\; if\; k\; is\; pair\\
\\\frac{2}{q}\; else.\end{array}\right.\]
\end{prop}

\vspace{1cm}

\hspace{-0.5cm}\textbf{\large Olivier Lablée}{\large \par}

\hspace{-0.5cm}

\hspace{-0.5cm}Université Grenoble 1-CNRS

\hspace{-0.5cm}Institut Fourier

\hspace{-0.5cm}UFR de Mathématiques

\hspace{-0.5cm}UMR 5582

\hspace{-0.5cm}BP 74 38402 Saint Martin d'Hères 

\hspace{-0.5cm}mail: \textcolor{blue}{olivier.lablee@ac-grenoble.fr}

\hspace{-0.5cm}http://www-fourier.ujf-grenoble.fr/\textasciitilde{}lablee/
\end{document}